\documentclass[12pt]{article}
\usepackage{amsmath,latexsym,amssymb}
\usepackage{enumerate}
\usepackage{amsthm}
\usepackage{epsfig,graphicx}
\usepackage{epic}

\newcommand\NN{\mathbb{N}}

\newcommand\ZZ{\mathbb{Z}}

\newcommand\CC{\mathcal{C}}
\newcommand\DD{\mathcal{D}}

\newcommand\QQ{\mathbb{Q}}
\newcommand\TT{\mathcal{T}}
\newcommand\UU{\mathcal{U}}
\newcommand\VV{\mathcal{V}}
\newcommand\eps{{\varepsilon}}

\renewcommand{\mod}{\operatorname{mod}}

\newtheorem{theorem}{Theorem}[section]

\newtheorem{example}[theorem]{Example}
\newtheorem{lemma}[theorem]{Lemma}
\newtheorem{remark}[theorem]{Remark}
\newtheorem{proposition}[theorem]{Proposition}

\newcommand{\address}{Address: Department of Mathematics, University of North Texas, 1155 Union Circle \#311430, Denton, TX 76203-5017, USA; E-mail: allaart@unt.edu}

\title{The infinite derivatives of Okamoto's self-affine functions: an application of $\beta$-expansions}
\author{Pieter C. Allaart \footnote{\address}}

\begin{document}

\maketitle

\begin{abstract}
Okamoto's one-parameter family of self-affine functions $F_a: [0,1]\to[0,1]$, where $0<a<1$, includes the continuous nowhere differentiable functions of Perkins ($a=5/6$) and Bourbaki/Katsuura ($a=2/3$), as well as the Cantor function ($a=1/2$). The main purpose of this article is to characterize the set of points at which $F_a$ has an infinite derivative. We compute the Hausdorff dimension of this set for the case $a\leq 1/2$, and estimate it for $a>1/2$. For all $a$, we determine the Hausdorff dimension of the sets of points where: (i) $F_a'=0$; and (ii) $F_a$ has neither a finite nor an infinite derivative. The upper and lower densities of the digit $1$ in the ternary expansion of $x\in[0,1]$ play an important role in the analysis, as does the theory of $\beta$-expansions of real numbers.

\bigskip
{\it AMS 2010 subject classification}: 26A27, 26A30 (primary),  28A78, 11A63 (secondary)

\bigskip
{\it Key words and phrases}: Continuous nowhere differentiable function; singular function; Cantor function; infinite derivative; ternary expansion; beta-expansion; Hausdorff dimension; Komornik-Loreti constant; Thue-Morse sequence.
\end{abstract}

\section{Introduction}

In 2005, H. Okamoto \cite{Okamoto} introduced and studied a one-parameter family of self-affine functions $\{F_a: 0<a<1\}$ on the interval $[0,1]$ defined as follows: Let $f_0(x)=x$, and inductively, for $n=0,1,2,\dots$, let $f_{n+1}$ be the unique continuous function which is linear on each interval $[j/3^{n+1},(j+1)/3^{n+1}]$ with $j\in\ZZ$ and satisfies, for $k=0,1,\dots,3^n-1$, the equations
\begin{gather*}
f_{n+1}(k/3^n)=f_n(k/3^n), \qquad f_{n+1}\big((k+1)/3^n\big)=f_n\big((k+1)/3^n\big),\\
f_{n+1}\big((3k+1)/3^{n+1}\big)=f_n(k/3^n)+a\left[f_n\big((k+1)/3^n\big)-f_n(k/3^n)\right],\\
f_{n+1}\big((3k+2)/3^{n+1}\big)=f_n(k/3^n)+(1-a)\left[f_n\big((k+1)/3^n\big)-f_n(k/3^n)\right].
\end{gather*}
(See Figure \ref{fig:construction}.)
The sequence $\{f_n\}$ thus defined converges uniformly on $[0,1]$. Let $F_a:=\lim_{n\to\infty}f_n$, so $F_a$ is a continuous function from the unit interval $[0,1]$ onto itself. The idea of this simple construction originated with Perkins \cite{Perkins}, who considered the case $a=5/6$ and proved that $F_{5/6}$ is nowhere differentiable. The case $2/3$ was similarly treated by Bourbaki \cite[p.~35, Problem 1-2]{Bourbaki} and later by Katsuura \cite{Katsuura}. As shown by Okamoto and Wunsch \cite{OkaWunsch}, $F_a$ is singular when $0<a\leq 1/2$ and $a\neq 1/3$; in particular, $F_{1/2}$ is the Cantor function. Note that $F_{1/3}(x)=x$. Figure \ref{fig:Okamoto-graphs} shows the graph of $F_a$ for several values of $a$.

\begin{figure}
\begin{center}
\begin{picture}(360,150)(0,15)
\put(30,20){\line(1,0){126}}
\put(30,20){\line(0,1){126}}
\put(30,146){\line(1,0){126}}
\put(156,20){\line(0,1){126}}
\put(27,15){\makebox(0,0)[tl]{$0$}}
\put(72,18){\line(0,1){4}}
\put(62,15){\makebox(0,0)[tl]{$1/3$}}
\put(114,18){\line(0,1){4}}
\put(104,15){\makebox(0,0)[tl]{$2/3$}}
\put(154,15){\makebox(0,0)[tl]{$1$}}
\put(17,25){\makebox(0,0)[tl]{$0$}}
\put(28,104){\line(1,0){4}}
\put(16,107){\makebox(0,0)[tl]{$a$}}
\put(28,62){\line(1,0){4}}
\put(-2,67){\makebox(0,0)[tl]{$1-a$}}
\put(19,150){\makebox(0,0)[tl]{$1$}}
\put(30,20){\line(1,2){42}}
\put(72,104){\line(1,-1){42}}
\put(114,62){\line(1,2){42}}
\put(60,122){\makebox(0,0)[tl]{$f_1$}}
\thicklines
\dottedline{4}(30,20)(156,146)
\thinlines
\put(230,20){\line(1,0){126}}
\put(230,20){\line(0,1){126}}
\put(230,146){\line(1,0){126}}
\put(356,20){\line(0,1){126}}
\put(227,15){\makebox(0,0)[tl]{$0$}}
\put(272,18){\line(0,1){4}}
\put(262,15){\makebox(0,0)[tl]{$1/3$}}
\put(314,18){\line(0,1){4}}
\put(304,15){\makebox(0,0)[tl]{$2/3$}}
\put(354,15){\makebox(0,0)[tl]{$1$}}
\put(217,25){\makebox(0,0)[tl]{$0$}}
\put(228,104){\line(1,0){4}}
\put(216,107){\makebox(0,0)[tl]{$a$}}
\put(228,62){\line(1,0){4}}
\put(198,67){\makebox(0,0)[tl]{$1-a$}}
\put(219,150){\makebox(0,0)[tl]{$1$}}
\put(230,20){\line(1,4){14}}
\put(244,76){\line(1,-2){14}}
\put(258,48){\line(1,4){14}}
\put(272,104){\line(1,-2){14}}
\put(286,76){\line(1,1){14}}
\put(300,90){\line(1,-2){14}}
\put(314,62){\line(1,4){14}}
\put(328,118){\line(1,-2){14}}
\put(342,90){\line(1,4){14}}
\put(260,122){\makebox(0,0)[tl]{$f_2$}}
\thicklines
\dottedline{4}(230,20)(272,104)
\dottedline{4}(272,104)(314,62)
\dottedline{4}(314,62)(356,146)
\end{picture}
\end{center}
\caption{The first two steps in the construction of $F_a$}
\label{fig:construction}
\end{figure}
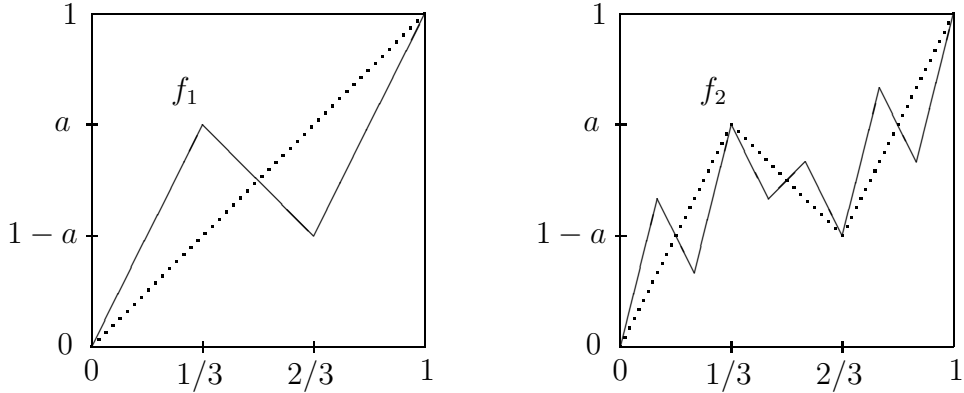

Let $a_0\approx .5592$ be the unique real root of $54a^3-27a^2=1$. Okamoto \cite{Okamoto} showed that (i) $F_a$ is nowhere differentiable if $2/3\leq a<1$; (ii) $F_a$ is nondifferentiable at almost every $x\in[0,1]$ but differentiable at uncountably many points if $a_0<a<2/3$; and (iii) $F_a$ is differentiable almost everywhere but nondifferentiable at uncountably many points if $0<a<a_0$. Okamoto left open the case $a=a_0$, but Kobayashi \cite{Kobayashi} later showed, using the law of the iterated logarithm, that $F_{a_0}$ is nondifferentiable almost everywhere. It is not difficult to see that, if $a\neq 1/3$ and $F_a$ has a finite derivative at $x$, then $F_a'(x)=0$; see Section \ref{sec:prelim}.

\begin{figure}
\begin{center}
\epsfig{file=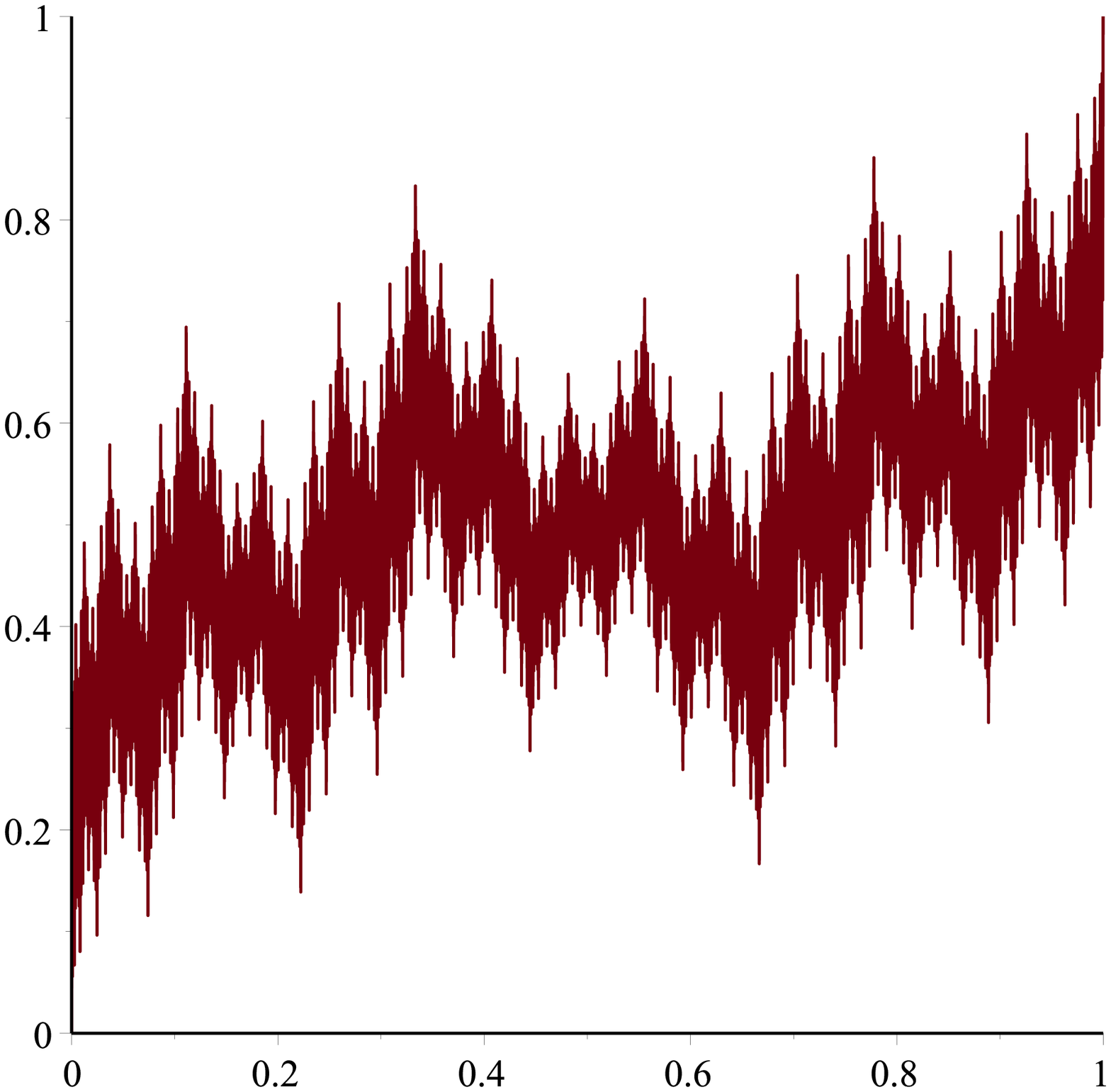, height=.3\textheight, width=.4\textwidth} \qquad\quad
\epsfig{file=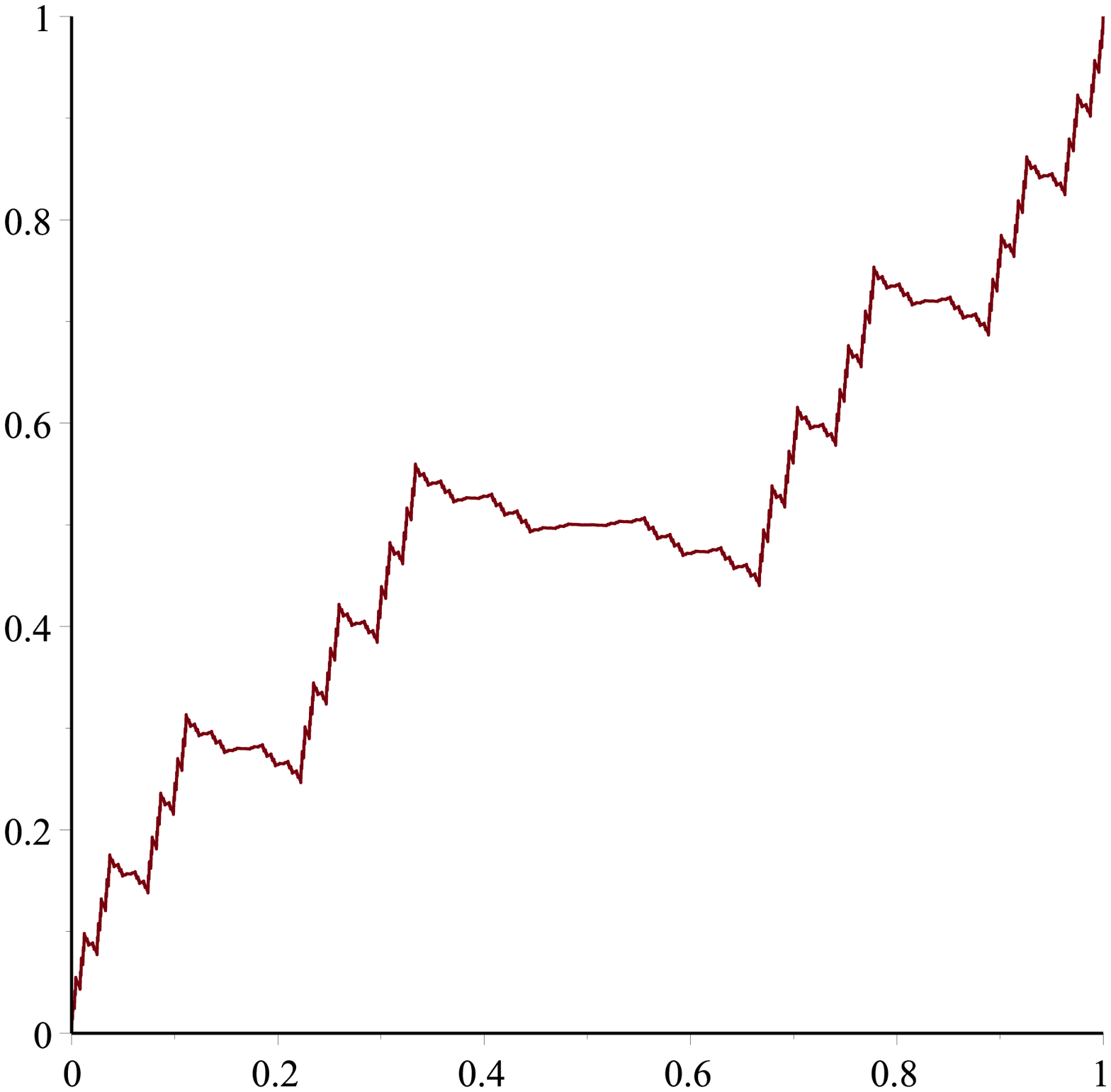, height=.3\textheight, width=.4\textwidth}\\
\epsfig{file=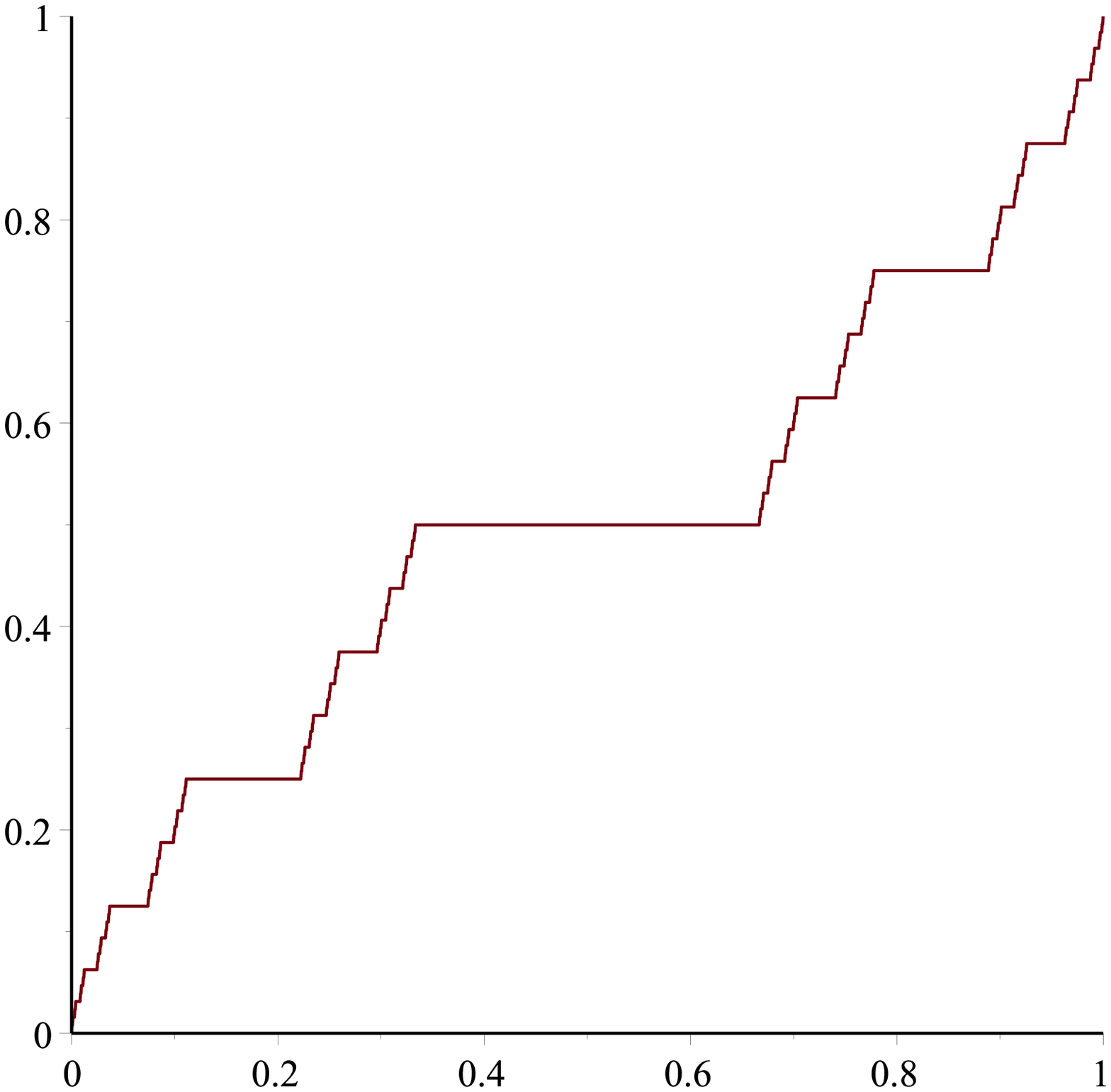, height=.3\textheight, width=.4\textwidth} \qquad\quad
\epsfig{file=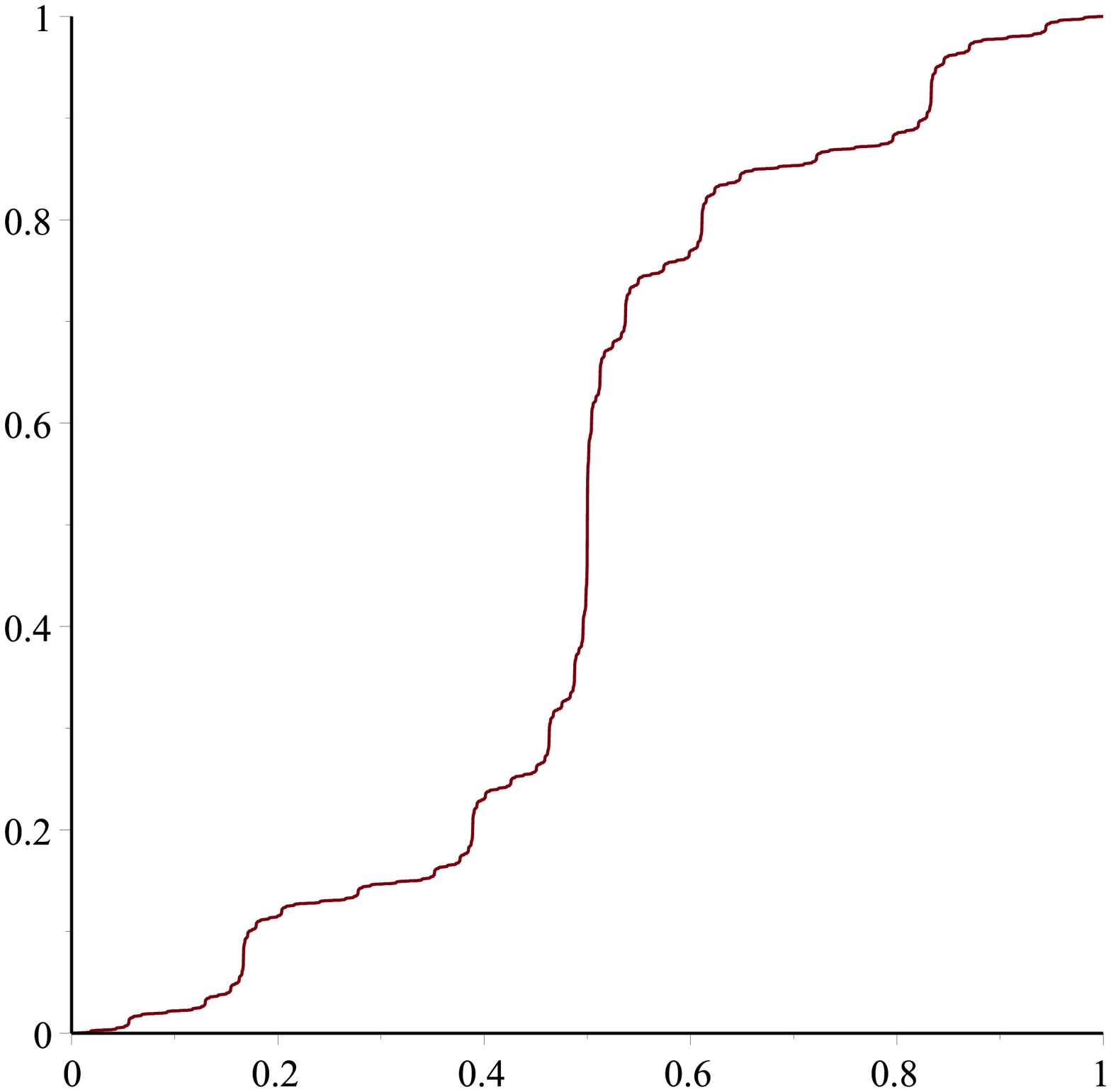, height=.3\textheight, width=.4\textwidth}
\caption{Graph of $F_a$ for several values of $a$. Top left: $a=5/6$ (Perkins' function); top right: $a=\hat{a}\approx .5598$; bottom left: $a=1/2$ (the Cantor function); bottom right: $a=1/5$ (a ``slippery devil's staircase").}
\label{fig:Okamoto-graphs}
\end{center}
\end{figure}

The main purpose of this article is to investigate the set $\DD_\infty(a)$ of points at which $F_a$ has an {\em infinite} derivative. In the parameter region $0<a<1/2$, where $F_a$ is strictly increasing, the situation is straightforward: $F_a'(x)=\infty$ if and only if $f_n'(x)\to\infty$. Since $f_n'(x)$ is readily expressed in terms of the ternary expansion of $x$, the Hausdorff dimension of $\DD_\infty(a)$ can be calculated for $a$ in this range by relating this set to certain sets defined in terms of the upper and lower frequency of the digit $1$ in the ternary expansion of $x\in(0,1)$. Using the same ideas we also obtain the Hausdorff dimensions of the exceptional sets in Okamoto's theorem; that is, the set of points where $F_a'(x)=0$ (for $a_0<a<2/3$), and the set of points where $F_a$ has neither a finite nor an infinite derivative (for $0<a<a_0$).

More interesting is the characterization of $\DD_\infty(a)$ in the parameter region $1/2<a<1$. Here $\DD_\infty(a)$ has strictly smaller Hausdorff dimension than the set $\{x:f_n'(x)\to\pm\infty\}$, though we are not able to compute the dimension of $\DD_\infty(a)$ exactly. Theorem \ref{thm:main} below gives a precise, though somewhat opaque, description of $\DD_\infty(a)$, which turns out to have surprising consequences. The condition for membership in $\DD_\infty(a)$ suggests a connection with $\beta$-expansions of real numbers, and indeed, we use the literature on $\beta$-expansions (e.g. \cite{GlenSid,JSS,Parry}) to show that $\DD_\infty(a)$ is (i) empty if $a\geq \rho:=(\sqrt{5}-1)/2\approx .6180$; (ii) countably infinite if $\hat{a}<a<\rho$; and (iii) uncountable with strictly positive Hausdorff dimension if $1/2<a<\hat{a}$. Here $\hat{a}\approx .5598$ is the reciprocal of the Komornik-Loreti constant, which is intimately related to the famous Thue-Morse sequence; see Section \ref{sec:prelim} below. Using very recent results on $\beta$-expansions by Kong and Li \cite{KongLi} and Komornik et al. \cite{KKL}, we conclude that the Hausdorff dimension of $\DD_\infty(a)$ is continuous on $1/2<a<\hat{a}$ and decreases on this interval in the manner of a ``reversed" devil's staircase.

In the boundary case $a=1/2$, we obtain Eidswick's \cite{Eidswick} characterization of $\DD_\infty(a)$ as a special case of our main theorem.

The condition for $F_a$ to have an infinite derivative at $x$ simplifies when $x$ is rational. We make this precise in the final section of the paper.


We mention here some other known results about Okamoto's functions. First, $F_a$ is self-affine: The portions of the graph above the intervals $[0,1/3]$, $[1/3,2/3]$ and $[2/3,1]$ are affine images of the whole graph, contracted horizontally by $1/3$ and vertically by $a$, $|2a-1|$ and $a$, respectively, with the middle one also being reflected vertically when $a>1/2$. As a result, the box-counting dimension of the graph of $F_a$ follows from Example 11.4 in \cite{Falconer}. It is $1$ if $a\leq 1/2$, and $1+\log_3(4a-1)$ if $a>1/2$. (Although the middle part of the graph is contracted more strongly in the vertical direction than in the horizontal direction when $a<2/3$, the formula given in \cite{Falconer} is nonetheless applicable since the affine maps do not include shears.) This result was also obtained by McCollum \cite{McCollum}, who claims the same value for the Hausdorff dimension of the graph. Unfortunately, his proof is incorrect. McCollum attempts to apply the mass distribution principle, and seems to use the mass distribution that assigns to each ``basic rectangle" a measure proportional to its height. But he then appears to incorrectly assume that within each basic rectangle the mass is uniformly distributed. In fact, the Hausdorff dimension of the graph of $F_a$ seems rather more difficult to determine than the box-counting dimension, and it would not be surprising if the two dimensions were different for certain values of $a$.


Second, a very interesting paper by Seuret \cite{Seuret} shows how $F_a$ can be expressed as the composition of a monofractal function and an increasing function, and also computes the multifractal spectrum of $F_a$. 

In recent years, there has been a great deal of interest in the non-differentiability sets of singular functions. Much of the research concerns two main classes of functions: Cantor-like functions, whose points of increase form a Lebesgue-null set, on the one hand; and strictly increasing singular functions on the other. Following \cite{KS1}, let us call functions in the former class ``ordinary devil's staircases", and functions in the latter class ``slippery devil's staircases". In most cases, formulas are given for the Hausdorff dimension of the sets $\Delta_\infty(f)$ and $\Delta_\sim(f)$ of points where the function $f$ in question has, respectively, an infinite derivative, or neither a finite nor an infinite derivative. For the class of ordinary devil's staircases, this work was begun by Darst \cite{Darst}, who showed for $f$ the classical Cantor function ($F_{1/2}$ in our notation) that $\dim_H \Delta_{\sim}(f)=(\log 2/\log 3)^2$. Successive generalizations were obtained by Darst \cite{Darst2}, Falconer \cite{Falconer2}, Kesseb\"ohmer and Stratmann \cite{KS1}, and, most recently, Troscheit \cite{Troscheit}. Slippery devil's staircases were examined by Jordan et al. \cite{JKPS} and Kesseb\"ohmer and Stratmann \cite{KS2}, among others. Qualitatively, the results in the above papers suggest the following dichotomy: $\dim_H \Delta_\sim(f)<\dim_H \Delta_\infty(f)$ for ordinary devil's staircases, while $\dim_H \Delta_\sim(f)=\dim_H \Delta_\infty(f)$ for slippery devil's staircases. This is illustrated by Okamoto's functions $F_a$, which, for $0<a<1/2$, belong to the family of conjugacies of interval maps considered in \cite{JKPS}. (See Theorem \ref{thm:size-and-dimension} below.)

Finally, the infinite derivatives of another famous continuous nowhere differentiable function, namely that of Takagi \cite{Takagi}, were characterized by the present author and Kawamura \cite{AK} and Kr\"uppel \cite{Kruppel}.

\section{Notation and main results} \label{sec:prelim}

The following notation is used throughout. The set of positive integers is denoted by $\NN$, and the set of nonnegative integers by $\ZZ_+$. For $x\in[0,1]$, the {\em ternary expansion} of $x$ is the sequence $\xi_1,\xi_2,\dots$ defined by $x=\sum_{n=1}^\infty \xi_n/3^n$, and $\xi_n\in\{0,1,2\}$ for all $n$. If $x$ has two ternary expansions we take the one ending in all $0$'s, except when $x=1$, in which case we take the expansion ending in all $2$'s. For $n\in\NN$, let $i(n):=\#\{j: 1\leq j\leq n, \xi_j=1\}$, so $i(n)$ is the number of $1$'s in the first $n$ ternary digits of $x$. When ambiguities may arise we write $\xi_n(x)$ instead of $\xi_n$, and $i(n;x)$ instead of $i(n)$. Let $N_1(x):=\sup_n i(n)$ be the total number of $1$'s in the ternary expansion of $x$. Denote by $\CC$ the ternary Cantor set in $[0,1]$. The Hausdorff dimension of a set $E$ will be denoted by $\dim_H E$; see \cite{Falconer} for the definition and properties.

For a function $h$, let $h^+$ and $h^-$ denote the right-hand and left-hand derivatives of $h$, respectively (assuming they exist). Note that
\begin{equation}
f_n^+(x)=3^n a^{n-i(n)}(1-2a)^{i(n)}, \qquad x\in[0,1),
\label{eq:fn-slope}
\end{equation}
where following standard convention we set $0^0=1$.

\begin{proposition} \label{prop:zero-or-infinite}
If $a\neq 1/3$ and $F_a$ has a finite derivative at $x$, then $F_a'(x)=0$.
\end{proposition}

\begin{proof}
Since $F_a(k/3^n)=f_n(k/3^n)$ for $n\in\NN$ and $k\in\{0,1,\dots,3^n\}$, it follows that if $F_a$ has a derivative (finite or infinite) at $x$, its value must be $F_a'(x)=\lim_{n\to\infty} f_n^+(x)$. If $a\not\in\{1/3,1/2\}$, then $f_{n+1}^+(x)/f_n^+(x)\in\{3a,3(1-2a)\}$ for each $n$, so $\lim_{n\to\infty}f_n^+(x)$, if it exists, can only equal $0$ or $\pm\infty$. If $a=1/2$, it is immediate from \eqref{eq:fn-slope} that $f_n^+(x)$ cannot converge to a positive and finite value. 
\end{proof}

The next proposition identifies situations where the derivative of $F_a$ behaves ``as expected". The first statement was included in \cite{Okamoto} without proof.

\begin{proposition} \label{prop:zero-derivative}
Let $x\in(0,1)$.
\begin{enumerate}[(i)]
\item If $a\neq 1/2$ and $f_n^+(x)\to 0$, then $F_a'(x)=0$.
\item If $0<a<1/2$ and $f_n^+(x)\to\infty$, then $F_a'(x)=\infty$.
\end{enumerate}
\end{proposition}

Proposition \ref{prop:zero-or-infinite} indicates a natural partition of $(0,1)$ into the three sets
\begin{gather*}
\DD_0(a):=\{x\in(0,1): F_a'(x)=0\},\\
\DD_\infty(a):=\{x\in(0,1): F_a'(x)=\pm\infty\},
\end{gather*}
and
\begin{equation*}
\mathcal{N}(a):=\{x\in(0,1): F_a\ \mbox{has no (finite or infinite) derivative at $x$}\}.
\end{equation*}
Let $\mathcal{L}$ denote Lebesgue measure on $(0,1)$. By Okamoto's theorem \cite[Theorem 4]{Okamoto}, $\mathcal{L}(\DD_0(a))=1$ for $0<a<a_0$, $a\neq 1/3$, and $\mathcal{L}(\DD_0(a))=0$ for $a\geq a_0$. From Proposition \ref{prop:zero-derivative} and \eqref{eq:fn-slope} it transpires that membership of a point $x$ in $\DD_0(a)$ is nearly determined by the (upper or lower) frequency of the digit $1$ in the ternary expansion of $x$. This enables us to compute $\dim_H\DD_0(a)$ when $a_0\leq a<2/3$, and similarly, $\dim_H\DD_\infty(a)$ for $0<a<1/2$, and $\dim_H\mathcal{N}(a)$ for all $a$. This is undertaken in Section \ref{sec:dimensions}.

By contrast, it turns out that when $a\geq 1/2$, $F_a$ may not have an infinite derivative at $x$ even if $\lim_{n\to\infty}f_n'(x)=\pm\infty$. In fact, we will see in Section \ref{sec:beta-expansions} that for $a>1/2$, $\dim_H\DD_\infty(a)$ is strictly smaller than that of $\{x\in[0,1]:f_n'(x)\to\pm\infty\}$. The main theorem below uses the following additional notation. For integers $j$ and $k$, let
\begin{equation*}
\delta_k(j):=\begin{cases}
1 &\mbox{if $j=k$},\\
0 &\mbox{if $j\neq k$}.
\end{cases}
\end{equation*}
For $d\in\{0,1,2\}$ and $n\in\NN$, let $r_n(d)$ denote the run length of the digit $d$ starting with the $(n+1)$th digit of $x$. That is,
\begin{equation*}
r_n(d):=\inf\{k>n: \xi_k\neq d\}-n-1.
\end{equation*}

\begin{theorem} \label{thm:main}
\begin{enumerate}[(i)]
\item Let $1/2<a<1$. Then $F_a'(x)=\pm\infty$ if and only if $N_1(x)<\infty$ and
\begin{equation}
(3a)^n\left(1-\sum_{k=1}^\infty a^{k}\delta_d(\xi_{n+k})\right)\to\infty, \qquad d=0,2,
\label{eq:main-condition}
\end{equation}
in which case $F_a'(x)=\infty$ if $N_1(x)$ is even, and $F_a'(x)=-\infty$ if $N_1(x)$ is odd.
\item Let $a=1/2$, and put $c:=\log_2 3-1$. Then $F_a'(x)=\infty$ if and only if $N_1(x)=0$ and
\begin{equation}
cn-r_n(d)\to\infty, \qquad d=0,2.
\label{eq:Cantor-function-condition}
\end{equation}
\end{enumerate}
\end{theorem}

In fact, we shall see in Section \ref{sec:proof} that condition \eqref{eq:main-condition} for $d=0$ (resp., $d=2$) is necessary in order for $F_a$ to have an infinite left-hand (resp., right-hand) derivative at $x$, and similarly for condition \eqref{eq:Cantor-function-condition}.

Note that (ii) specifies the points of infinite derivative of the Cantor function. Such a characterization was given previously by Eidswick \cite[Corollary 1]{Eidswick}, who showed, as a consequence of a more general theorem, that $F_{1/2}'(x)=\infty$ if and only if
\begin{equation}
\lim_{n\to\infty}\frac{3^{z(n)}}{2^{z(n+1)}}=\lim_{n\to\infty}\frac{3^{t(n)}}{2^{t(n+1)}}=\infty,
\label{eq:Eidswick-condition}
\end{equation}
where $z(n)$ denotes the position of the $n$th 0, and $t(n)$ the position of the $n$th 2, in the ternary expansion of $x$. (Of course, Eidswick did not use the notation $F_{1/2}$.) The equivalence of \eqref{eq:Cantor-function-condition} and \eqref{eq:Eidswick-condition} can be seen by taking logarithms in \eqref{eq:Eidswick-condition} and using the relationships $z(n+1)-z(n)=r_{z(n)}(2)-1$, $t(n+1)-t(n)=r_{t(n)}(0)-1$. We derive (ii) here quickly as a special case of (i).

\begin{remark} \label{rem:simpler-condition}
{\rm
Since $(3a)^n\to\infty$ when $a>1/2$, it is sufficient for \eqref{eq:main-condition} that 
\begin{equation*}
\limsup_{n\to\infty} \sum_{k=1}^\infty a^{k}\delta_d(\xi_{n+k})<1, \qquad d=0,2,
\end{equation*}
and necessary that $\limsup_{n\to\infty} \sum_{k=1}^\infty a^{k}\delta_d(\xi_{n+k})\leq 1$ for $d=0,2$. An interesting question, which the author has been unable to answer, is whether there exist values of $a$ and ternary sequences $\{\xi_n\}$ such that $\limsup_{n\to\infty} \sum_{k=1}^\infty a^{k}\delta_d(\xi_{n+k})=1$ but \eqref{eq:main-condition} holds for $d=0$ or $d=2$.
}
\end{remark}

\begin{example}
{\rm
Let $x=0.022(02)022(02)^2 022(02)^3 \dots 022(02)^n \dots$. Then
\begin{equation*}
\limsup_{n\to\infty} \sum_{k=1}^\infty a^{k}\delta_2(\xi_{n+k})=a+a^2+a^4+a^6+\dots=a+\frac{a^2}{1-a^2},
\end{equation*}
and this is less than $1$ if and only if $a+2a^2-a^3<1$. On the other hand,
\begin{equation*}
\limsup_{n\to\infty} \sum_{k=1}^\infty a^{k}\delta_0(\xi_{n+k})=a+a^3+a^5+\dots=\frac{a}{1-a^2}<a+\frac{a^2}{1-a^2}.
\end{equation*}
Hence, the condition for $d=2$ is more stringent. Let $a^*(x)\approx .5550$ be the unique root in $(0,1)$ of $a+2a^2-a^3=1$. By Remark \ref{rem:simpler-condition}, $F_a'(x)=\infty$ for $1/2<a<a^*(x)$, but $x\not\in \DD_\infty(a)$ when $a>a^*(x)$, despite the fact that $f_n'(x)=(3a)^n\to\infty$ for every $a>1/3$. (In fact, $F_a'(x)=\infty$ for $1/3<a<a^*(x)$. For $1/3<a<1/2$ this follows from Proposition \ref{prop:zero-derivative}(ii), and for $a=1/2$ it follows from Theorem \ref{thm:main}(ii), since $r_n(d)\leq 2$ for all $n$ and $d\in\{0,2\}$). The author conjectures that $F_a'(x)=\infty$ also when $a=a^*(x)$.
}
\end{example}

We next examine the size of $\DD_\infty(a)$ for $1/2<a<1$. Let $\rho:=(\sqrt{5}-1)/2\approx .6180$ be the golden ratio, and recall that the {\em Thue-Morse sequence} is the sequence $(t_j)_{j=0}^\infty$ of $0$'s and $1$'s given by $t_j=s_j \mod 2$, where $s_j$ is the number of $1$'s in the binary representation of $j$. Thus,
\begin{equation}
(t_j)_{j=0}^\infty=0110\ 1001\ 1001\ 0110\ 1001\ 0110\ 0110\ 1001\ \dots
\label{eq:Thue-Morse}
\end{equation}
Let $\hat{a}\approx .5598$ be the unique root in $(0,1)$ of the equation $\sum_{j=1}^\infty t_j a^j=1$. The reciprocal of $\hat{a}$ is known as the {\em Komornik-Loreti constant}, introduced in \cite{KomLor}. 

\begin{theorem} \label{cor:golden-ratio}
The set $\DD_\infty(a)$ is:
\begin{enumerate}[(i)]
\item empty if $a\geq\rho$;
\item countably infinite, containing only rational points, if $\hat{a}<a<\rho$;
\item uncountable with strictly positive Hausdorff dimension if $1/2<a<\hat{a}$.
\end{enumerate}
Moreover, on the interval $1/2<a<\hat{a}$, the function $a\mapsto \dim_H\DD_\infty(a)$ is continuous and nonincreasing in the manner of a devil's staircase; that is, there is a countable family $(I_j)_{j\in\NN}$ of disjoint subintervals of $(1/2,\hat{a})$ whose union has full Lebesgue measure in $(1/2,\hat{a})$ and such that $\dim_H\DD_\infty(a)$ is constant on $I_j$ for each $j$.
\end{theorem}

This result is a consequence of Theorem \ref{thm:main} and the literature on $\beta$-expansions of real numbers \cite{GlenSid,JSS,KKL,KongLi,Parry}. The idea is that the set $\DD_\infty(a)$ is very closely related to the set of points which have a unique $\beta$-expansion, where $\beta=1/a$. To give the reader a flavor of the arguments, we show here that $\DD_\infty(a)\neq\emptyset$ if and only if $a<\rho$ and $a\neq 1/3$. The remainder of Theorem \ref{cor:golden-ratio} is proved in Section \ref{sec:beta-expansions}.

\begin{proof}[Partial proof of Theorem \ref{cor:golden-ratio}]
Suppose $a\geq\rho$. Then $a+a^2\geq 1$, so condition \eqref{eq:main-condition} clearly fails if the ternary expansion of $x$ contains either $00$ or $22$ infinitely often. This leaves points with ternary expansions ending in $(20)^\infty$. But for such points, 
\begin{equation*}
\sum_{k=1}^\infty a^{k}\delta_2(\xi_{n+k})=a+a^3+a^5+\dots=\frac{a}{1-a^2}\geq 1
\end{equation*}
for infinitely many $n$, so \eqref{eq:main-condition} fails again.

On the other hand, if $a<\rho$, then $a/(1-a^2)<1$, and so any point $x$ whose ternary expansion ends in $(20)^\infty$ satisfies \eqref{eq:main-condition} in view of Remark \ref{rem:simpler-condition}.
\end{proof}

\begin{remark}
{\rm
{\em (a)} In fact, a fairly explicit description of points in $\DD_\infty(a)$ can be given when $\hat{a}<a<\rho$. For example, if $a$ is such that $a+a^2<1\leq a+a^2+a^4$, then $\DD_\infty(a)$ consists {\em exactly} of those points whose ternary expansion ends in $(20)^\infty$, as ternary expansions containing one of the words $222$, $000$, $2202$ or $0020$ infinitely often will be forbidden, as are expansions ending in $(2200)^\infty$. This simple combinatorial idea illustrates Theorem \ref{cor:golden-ratio}(ii); we elaborate on it further in the proof of Theorem \ref{cor:golden-ratio}, in Section \ref{sec:beta-expansions}.

{\em (b)} It is unclear whether $\DD_\infty(\hat{a})$ is countable or uncountable, but it will be shown in Remark \ref{rem:critical-value} that its Hausdorff dimension is zero.

{\em (c)} It is interesting to observe that, for $\rho\leq a<2/3$, $F_a$ has a finite derivative at infinitely many points but an infinite derivative nowhere.
}
\end{remark}

To end this section, we mention that triadic rational points in $(0,1)$ (i.e. points in the set $\TT:=\{j/3^n: n\in\NN,\ j=1,2,\dots,3^n-1\}$) are of some special interest. At such points, depending on the value of $a$, $F_a$ may have a vanishing derivative, an infinite derivative, a cusp, or a ``cliff" (with one one-sided derivative equal to zero and the other equal to $\infty$): 

\begin{proposition} \label{thm:triadic}
Let $x\in\mathcal{T}$.
\begin{enumerate}[(i)]
\item If $1/2<a<1$, then $F_a$ has a cusp at $x$; that is, either $F_a^+(x)=-F_a^-(x)=\infty$ (if $N_1(x)$ is even), or $F_a^+(x)=-F_a^-(x)=-\infty$  (if $N_1(x)$ is odd).
\item If $a=1/2$, then either $F_a'(x)=0$, or $F_a^+(x)=\infty$ and $F_a^-(x)=0$, or $F_a^+(x)=0$ and $F_a^-(x)=\infty$.
\item If $1/3<a<1/2$, then $F_a'(x)=\infty$.
\item If $0<a<1/3$, then $F_a'(x)=0$.
\end{enumerate}
Moreover, 
\begin{equation}
F_a^+(0)=F_a^-(1)=\begin{cases}
\infty & \mbox{if $a>1/3$},\\
0 & \mbox{if $a<1/3$}.
\end{cases}
\label{eq:endpoints}
\end{equation}
\end{proposition}

The remainder of this article is organized as follows. Proposition \ref{prop:zero-derivative}, Theorem \ref{thm:main} and Proposition \ref{thm:triadic} are proved in Section \ref{sec:proof}. In Section \ref{sec:dimensions} we compute the Hausdorff dimensions of $\DD_0(a)$ and $\mathcal{N}(a)$, and that of $\DD_\infty(a)$ for $0<a\leq 1/2$. In Section \ref{sec:beta-expansions} we review basic facts about $\beta$-expansions and prove Theorem \ref{cor:golden-ratio}. Finally, in Section \ref{sec:rational}, we simplify the condition \eqref{eq:main-condition} for the case of rational $x$, using ideas from Section \ref{sec:beta-expansions}.

\section{Vanishing and infinite derivatives} \label{sec:proof}

In this section we prove Proposition \ref{prop:zero-derivative}, Theorem \ref{thm:main} and Proposition \ref{thm:triadic}.
We use two key observations. First, for any triadic interval $[u_n,v_n]=[j/3^n,(j+1)/3^n]$ (where $n\in\NN$ and $j=0,1,\dots,3^n-1$),
\begin{equation}
u_n\leq x\leq v_n \quad\Rightarrow \quad \min\{F_a(u_n),F_a(v_n)\}\leq F_a(x)\leq\max\{F_a(u_n),F_a(v_n)\}.
\label{eq:graph-inside-rectangle}
\end{equation}
Second, if $a\neq 1/2$ and $s_{n,j}$ denotes the slope of $f_n$ on $[j/3^n,(j+1)/3^n]$, then
\begin{equation}
\frac{s_{n,j+1}}{s_{n,j}}\in\left\{\frac{a}{1-2a},\frac{1-2a}{a}\right\}, \qquad j=0,1,\dots,3^n-1.
\label{eq:consecutive-slopes-ratio}
\end{equation}
This can be checked by using an inductive argument. 

\begin{proof}[Proof of Proposition \ref{prop:zero-derivative}]
(i) Fix $a\in(0,1)\backslash \{1/2\}$, and suppose $f_n^+(x)\to 0$. Given $h>0$, let $n$ be the integer such that $3^{-n-1}<h\leq 3^{-n}$. Let $u_n=(j-1)/3^n$, $v_n=j/3^n$ and $w_n=(j+1)/3^n$, where $j\in\ZZ$ and $u_n\leq x<v_n$. Then $x+h<w_n$, so a double application of \eqref{eq:graph-inside-rectangle} gives
\begin{align*}
|F_a(x+h)-F_a(x)|&\leq |F_a(v_n)-F_a(u_n)|+|F_a(w_n)-F_a(v_n)|\\
&=3^{-n}|f_n^+(x)|+3^{-n}|f_n^+(v_n)|\leq 3^{-n}(1+C)|f_n^+(x)|,
\end{align*}
where $C=\max\{a/|2a-1|,|2a-1|/a\}$, and the last inequality follows from \eqref{eq:consecutive-slopes-ratio}. Since $h>3^{-n-1}$, we obtain
\begin{equation*}
\left|\frac{F_a(x+h)-F_a(x)}{h}\right|\leq 3(1+C)|f_n^+(x)|,
\end{equation*}
and hence, by Proposition \ref{prop:zero-or-infinite}, $F_a^+(x)=0$. Now \eqref{eq:consecutive-slopes-ratio} implies that $f_n^-(x)\to 0$ as well, so by symmetry, $F_a^-(x)=0$. Thus, $F_a'(x)=0$.

(ii) The second statement follows from the more general result below by taking $K=3$ and $C=\max\{a/(1-2a),(1-2a)/a\}$.
\end{proof}

\begin{lemma} \label{lem:when-infinite-derivative}
Let $K>1$ be an integer. Let $\{g_n\}$ be a sequence of strictly increasing continuous functions on $[0,1]$ such that (i) $g_n$ is linear in $(j/K^n,(j+1)/K^n)$ for $j=0,1,\dots,K^n-1$; (ii) $g_{n+1}(j/K^n)=g_n(j/K^n)$ for all $n\in\NN$ and $j\in\{0,1,\dots,K^n\}$; and (iii) $g_n$ converges pointwise in $[0,1]$ to a function $g$. Let $s_{n,j}:=g_n^+(j/K^n)$, and suppose there is a constant $C>1$ such that
\begin{equation}
C^{-1}\leq\frac{s_{n,j+1}}{s_{n,j}}\leq C \qquad \mbox{for all $n$ and all $j$}.
\label{eq:uniform-boundedness}
\end{equation}
Then for $x\in(0,1)$, $g'(x)=\infty$ if and only if $g_n^+(x)\to\infty$.
\end{lemma}

\begin{proof}
Fix $x\in(0,1)$ and suppose $g_n^+(x)\to\infty$. Given $h>0$, let $n\in\NN$ such that $K^{-n-1}<h\leq K^{-n}$, and let $j$ be the integer such that $(j-1)/K^{n+2}<x\leq j/K^{n+2}$. Then $x+h>(j+1)/K^{n+2}$, and since $g$ is nondecreasing,
\begin{align*}
g(x+h)-g(x)&\geq g\left(\frac{j+1}{K^{n+2}}\right)-g\left(\frac{j}{K^{n+2}}\right)\\
&=K^{-(n+2)}g_{n+2}^+\left(\frac{j}{K^{n+2}}\right)\geq C^{-1}K^{-(n+2)}g_{n+2}^+(x),
\end{align*}
so that
\begin{equation*}
\frac{g(x+h)-g(x)}{h}\geq K^n\big(g(x+h)-g(x)\big)\geq C^{-1}K^{-2}g_{n+2}^+(x).
\end{equation*}
This shows that $g^+(x)=\infty$. Since \eqref{eq:uniform-boundedness} implies that $g_n^-(x)\geq C^{-1}g_n^+(x)$ for all $n$, a similar argument gives $g^-(x)=\infty$. Thus, $g'(x)=\infty$. 

Conversely, suppose $g'(x)=\infty$. By hypothesis (ii) of the lemma, $g(j/K^n)=g_n(j/K^n)$ for all $n\in\NN$ and $j\in\{0,1,\dots,K^n\}$. Therefore, $g_n^+(x)$ must tend to $\infty$.
\end{proof}

The next lemma and its proof represent the core of the investigation of the infinite derivatives of $F_a$.

\begin{lemma} \label{lem:infinite-right-derivative}
Let $1/2\leq a<1$. Let $x\in[0,1)$ with ternary expansion $\{\xi_n\}$ and assume $\xi_n\in\{0,2\}$ for each $n$. Then $F_a^+(x)=\infty$ if and only if
\begin{equation}
(3a)^n\left[1-\sum_{k=1}^\infty a^k\delta_2(\xi_{n+k})\right]\to\infty.
\label{eq:d-is-2}
\end{equation}
\end{lemma}

\begin{proof}
We use the following explicit expression for $F_a(x)$ (see \cite{Kobayashi}): 
\begin{equation*}
F_a(x)=\sum_{k=1}^\infty a^{k-1-i(k-1)}(1-2a)^{i(k-1)}q(\xi_k),
\end{equation*}
where $q(0)=0$, $q(1)=a$ and $q(2)=1-a$. Since we assume here that $\xi_n\in\{0,2\}$ for each $n$, this simplifies to
\begin{equation}
F_a(x)=\sum_{k=1}^\infty a^{k-1}(1-a)\delta_2(\xi_k).
\label{eq:simpler-explicit-expression}
\end{equation}

Suppose first that $F_a^+(x)=\infty$. For $n\in\NN$, let $x_n:=(j+1)/3^n$, where $j$ is the integer such that $(j-1)/3^n\leq x<j/3^n$. Clearly,
\begin{equation}
\frac{F_a(x_n)-F_a(x)}{x_n-x}\to\infty.
\label{eq:special-right-sequence}
\end{equation}
Fix $n$. If $\xi_n=0$, then $x_n=0.\xi_1\xi_2\dots\xi_{n-1}200\dots$, so \eqref{eq:simpler-explicit-expression} gives
\begin{align}
\begin{split}
F_a(x_n)-F_a(x)
&=a^{n-1}(1-a)-\sum_{k=n+1}^\infty a^{k-1}(1-a)\delta_2(\xi_k)\\
&=a^{n-1}(1-a)\left[1-\sum_{k=1}^\infty a^k\delta_2(\xi_{n+k})\right].
\end{split}
\label{eq:difference-calculation}
\end{align}
This expression results also when $\xi_n=2$, because regardless of whether $\xi_n=0$ or $2$, the slope of $f_n$ on $[(j-1)/3^n,j/3^n]$ is $(3a)^n$, and the slope of $f_n$ on $[j/3^n,(j+1)/3^n]$ is $3^n a^{n-1}(1-2a)$ in view of \eqref{eq:consecutive-slopes-ratio}. Since $1/3^n<x_n-x\leq 2/3^n$, it follows from \eqref{eq:difference-calculation} that \eqref{eq:special-right-sequence} is equivalent to \eqref{eq:d-is-2}.

Conversely, suppose we have \eqref{eq:d-is-2}. Given $h>0$, let $n\in\NN$ be such that $3^{-n-1}<h\leq 3^{-n}$, let $j$ be the integer such that $(j-1)/3^n\leq x<j/3^n$, and define $x_n$ as above. Then \eqref{eq:special-right-sequence} holds, so in particular $F_a(x_n)>F_a(x)$ for all sufficiently large $n$. Since $f_n'=(3a)^n>0$ on $((j-1)/3^n,j/3^n)$, \eqref{eq:consecutive-slopes-ratio} implies that $f_n'\leq 0$ on $(j/3^n,(j+1)/3^n)$ (with equality if $a=1/2$). Thus, if $x+h\geq j/3^n$, we have immediately from \eqref{eq:graph-inside-rectangle} that
\begin{equation*}
\frac{F_a(x+h)-F_a(x)}{h}\geq \frac{F_a(x_n)-F_a(x)}{h}\geq \frac{F_a(x_n)-F_a(x)}{x_n-x},
\end{equation*}
for $n$ large enough.

On the other hand, if $x+h<j/3^n$, then $\xi_{n+1}=0$ by the hypothesis of the lemma, so $(j-1)/3^{n}\leq x<(3j-2)/3^{n+1}$. Now $f_{n+1}'>0$ on the intervals $((j-1)/3^{n},(3j-2)/3^{n+1})$ and $((3j-1)/3^{n+1},j/3^n)$, and $f_{n+1}'\leq 0$ on $((3j-2)/3^{n+1},(3j-1)/3^{n+1})$. Thus, again by \eqref{eq:graph-inside-rectangle}, $F_a(x+h)\geq F_a((3j-1)/3^{n+1})=F_a(x_{n+1})$. Since $x_{n+1}-x>3^{-n-1}\geq h/3$, it follows that
\begin{equation*}
\frac{F_a(x+h)-F_a(x)}{h}\geq \frac{F_a(x_{n+1})-F_a(x)}{h}\geq \frac{F_a(x_{n+1})-F_a(x)}{3(x_{n+1}-x)},
\end{equation*}
for sufficiently large $n$. Thus, by \eqref{eq:special-right-sequence}, $F_a^+(x)=\infty$. 
\end{proof}

\begin{proof}[Proof of Theorem \ref{thm:main}]
Fix $x\in(0,1)\backslash \mathcal{T}$. (The case $x\in\TT$ is addressed in the proof of Proposition \ref{thm:triadic} below.) Observe that it is sufficient to determine whether $F_a$ has an infinite right-hand derivative at $x$: Since $F_a(1-x)=1-F_a(x)$, it follows that $F_a^-(x)=F_a^+(1-x)$ when at least one of these quantities exists, so the results for an infinite left-hand derivative follow by interchanging $0$'s and $2$'s in the ternary expansion of $x$.

Assume first that $a>1/2$. It is clear from \eqref{eq:fn-slope} and \eqref{eq:graph-inside-rectangle} that $F_a^+(x)$ can not be infinite if $\xi_n=1$ for infinitely many $n$, so we need only consider the case when $m:=N_1(x)<\infty$. If $m=0$, then \eqref{eq:fn-slope} and \eqref{eq:graph-inside-rectangle} imply that $F_a^+(x)$ cannot take the value $-\infty$, and by Lemma \ref{lem:infinite-right-derivative}, $F_a^+(x)=\infty$ if and only if \eqref{eq:main-condition} holds for $d=2$. Suppose now that $m>0$. Choose $n_0\in\NN$ so that $\xi_n\in\{0,2\}$ for all $n\geq n_0$. Let $j$ be the integer such that $j/3^{n_0}\leq x<(j+1)/3^{n_0}$, and put $\tilde{x}:=j/3^{n_0}$. Now we can write $x=\tilde{x}+3^{-n_0}x'$, where $N_1(\tilde{x})=N_1(x)=m$, and $x'\in[0,1)$ satisfies the hypothesis of Lemma \ref{lem:infinite-right-derivative}. Observe that \eqref{eq:d-is-2} holds for $x'$ if and only if it holds for $x$, because the condition is invariant under a shift of the sequence $\{\xi_n\}$. The graph of $F_a$ above the interval $I_0:=[j/3^{n_0},(j+1)/3^{n_0}]=[\tilde{x},\tilde{x}+3^{-n_0}]$ is an affine copy of the whole graph of $F_a$. Specifically, for $z\in I_0$,
\begin{equation*}
F_a(z)=F_a(\tilde{x})+3^{n_0}f_{n_0}^+(x)F_a(z'),
\end{equation*}
where $z'\in[0,1]$ is such that $z=\tilde{x}+3^{-n_0}z'$. This shows that $F_a^+(x)=f_{n_0}^+(x)F_a^+(x')$ when the quantity on the right hand side exists. Since $f_{n_0}^+$ is positive on $I_0$ if $m$ is even, and negative if $m$ is odd, we conclude that $F_a$ has an infinite derivative at $x$ if and only if \eqref{eq:d-is-2} holds, in which case $F_a^+(x)=\infty$ if $m$ is even, and $F_a^+(x)=-\infty$ if $m$ is odd.

Next, assume $a=1/2$. In order for $F_a^+(x)$ to be infinite, it is necessary that $\xi_k\in\{0,2\}$ for all $k$, in view of \eqref{eq:fn-slope}. Assuming this, Lemma \ref{lem:infinite-right-derivative} implies that $F_a^+(x)=\infty$ if and only if \eqref{eq:d-is-2} holds (with $a=1/2$). Since
\begin{equation*}
1-\left(\frac12\right)^{r_n(2)}\leq \sum_{k=1}^\infty \left(\frac12\right)^k \delta_2(\xi_{k+n})\leq 1-\left(\frac12\right)^{r_n(2)+1},
\end{equation*}
this is the case if and only if
\begin{equation*}
3^n\left(\frac12\right)^{n+r_n(2)}\to\infty,
\end{equation*}
and taking logarithms, this reduces to the case $d=2$ in \eqref{eq:Cantor-function-condition}.
\end{proof}

\begin{proof}[Proof of Proposition \ref{thm:triadic}]
Fix $x\in\mathcal{T}$. Assume first that $a>1/2$. Since $\xi_n=0$ for all sufficiently large $n$, \eqref{eq:d-is-2} clearly holds, and by the argument in the proof of Theorem \ref{thm:main}, $F_a$ has an infinite right derivative at $x$. Applying this to $1-x$ shows (via the relation $F_a^-(x)=F_a^+(1-x)$) that $F_a$ also has an infinite left derivative at $x$. By \eqref{eq:consecutive-slopes-ratio}, $f_n^+(x)$ and $f_n^-(x)$ have opposite signs for all sufficiently large $n$, and hence, so do $F_a^+(x)$ and $F_a^-(x)$. This proves (i).

Next, let $a=1/2$. If $x$ lies in the interior of one of the removed intervals in the construction of the ternary Cantor set $\CC$, then $F_a'(x)=0$. Otherwise, $x$ is an endpoint of a removed interval, say it is a right endpoint. Then $F_a^-(x)=0$, and $\xi_n\in\{0,2\}$ for all $n$, so by Lemma \ref{lem:infinite-right-derivative}, $F_a^+(x)=\infty$. By symmetry, if $x$ is the left endpoint of a removed interval, then $F_a^+(x)=0$ and $F_a^-(x)=\infty$. This establishes (ii).

Statements (iii) and (iv) follow directly from Proposition \ref{prop:zero-derivative}. Essentially the same arguments establish \eqref{eq:endpoints}.
\end{proof}

\section{Frequency of digits and Hausdorff dimension} \label{sec:dimensions}

In this section we determine the Hausdorff dimensions of the sets $\DD_0(a)$ and $\mathcal{N}(a)$, as well as that of $\DD_\infty(a)$ for $0<a\leq 1/2$. We also examine how these sets vary with the parameter $a$.

Define the auxiliary functions
\begin{equation*}
\phi(a):=\frac{\log(3a)}{\log a-\log|2a-1|}, \qquad a\in(0,2/3]\backslash\{1/3,1/2\},
\end{equation*}
and
\begin{equation*}
h(p):=\frac{-p\log p-(1-p)\log(1-p)+(1-p)\log 2}{\log 3}, \qquad 0\leq p\leq 1,
\end{equation*}
where, following standard convention, we set $0\log 0\equiv 0$. We extend $\phi$ continuously to $[0,2/3]$ by setting $\phi(0):=\lim_{a\downarrow 0}\phi(a)=1$, $\phi(1/3):=\lim_{a\to 1/3}\phi(a)=1/3$, and $\phi(1/2):=\lim_{a\to 1/2}\phi(a)=0$. Note that $\phi(2/3)=1$. It can be shown that $\phi$ is strictly decreasing on $[0,1/2]$, and strictly increasing on $[1/2,2/3]$. Note that $h$ is maximized at $p=1/3$, with $h(1/3)=1$. See Figure \ref{fig:phi-and-h} for graphs of $\phi$ and $h$. Finally, let
\begin{equation*}
d(a):=h(\phi(a)), \qquad 0\leq a\leq 2/3.
\end{equation*}
The graph of $d$ is shown in Figure \ref{fig:d}. Note that, since $\phi(a_0)=1/3$, $d(a)$ attains its maximum value of $1$ at both $a=1/3$ and $a=a_0$.
(Recall from the Introduction that $a_0\approx .5592$ is the unique real root of $54a^3-27a^2=1$.)

\begin{figure}
\begin{center}
\epsfig{file=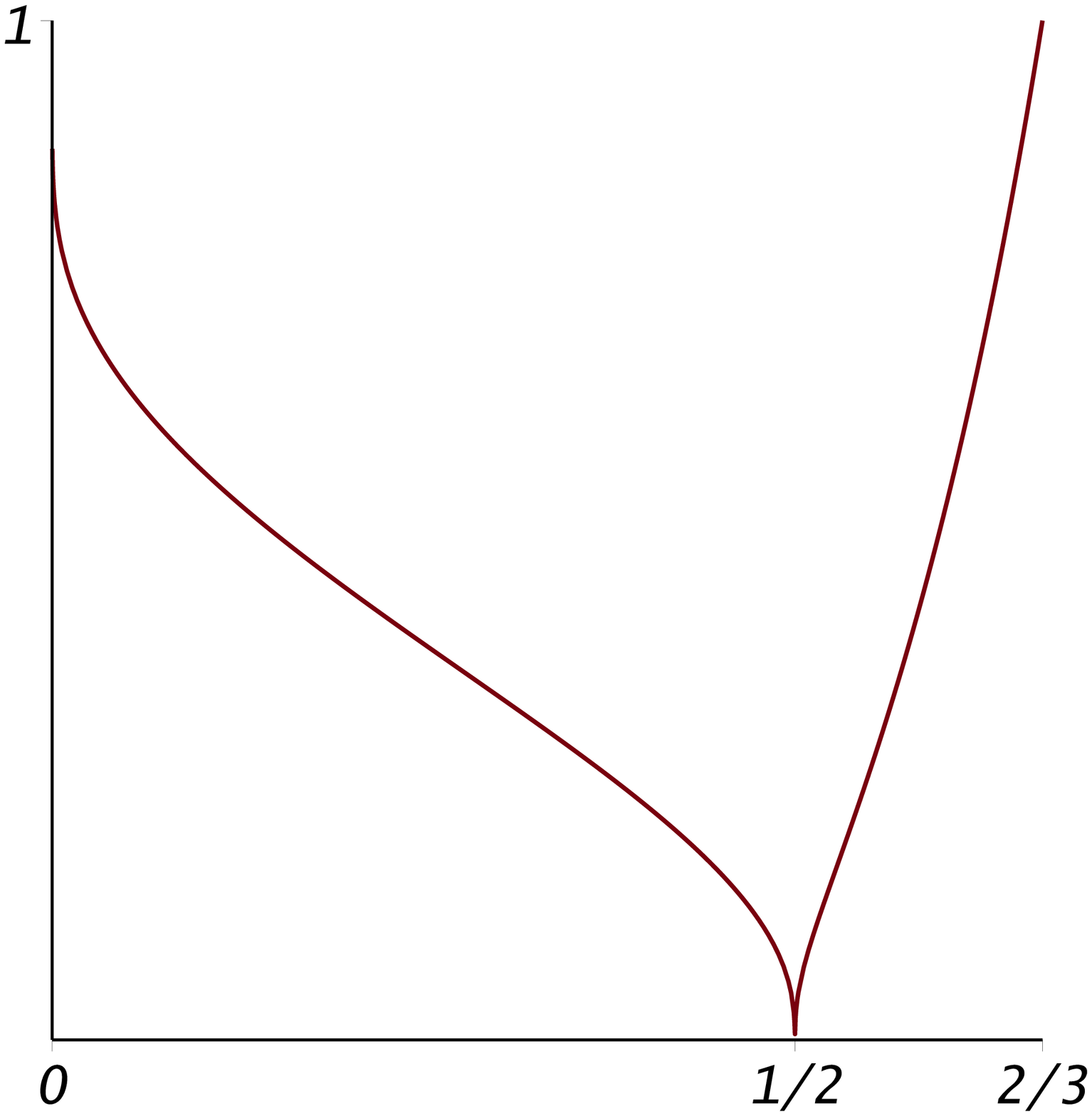, height=.2\textheight, width=.35\textwidth} \qquad\qquad
\epsfig{file=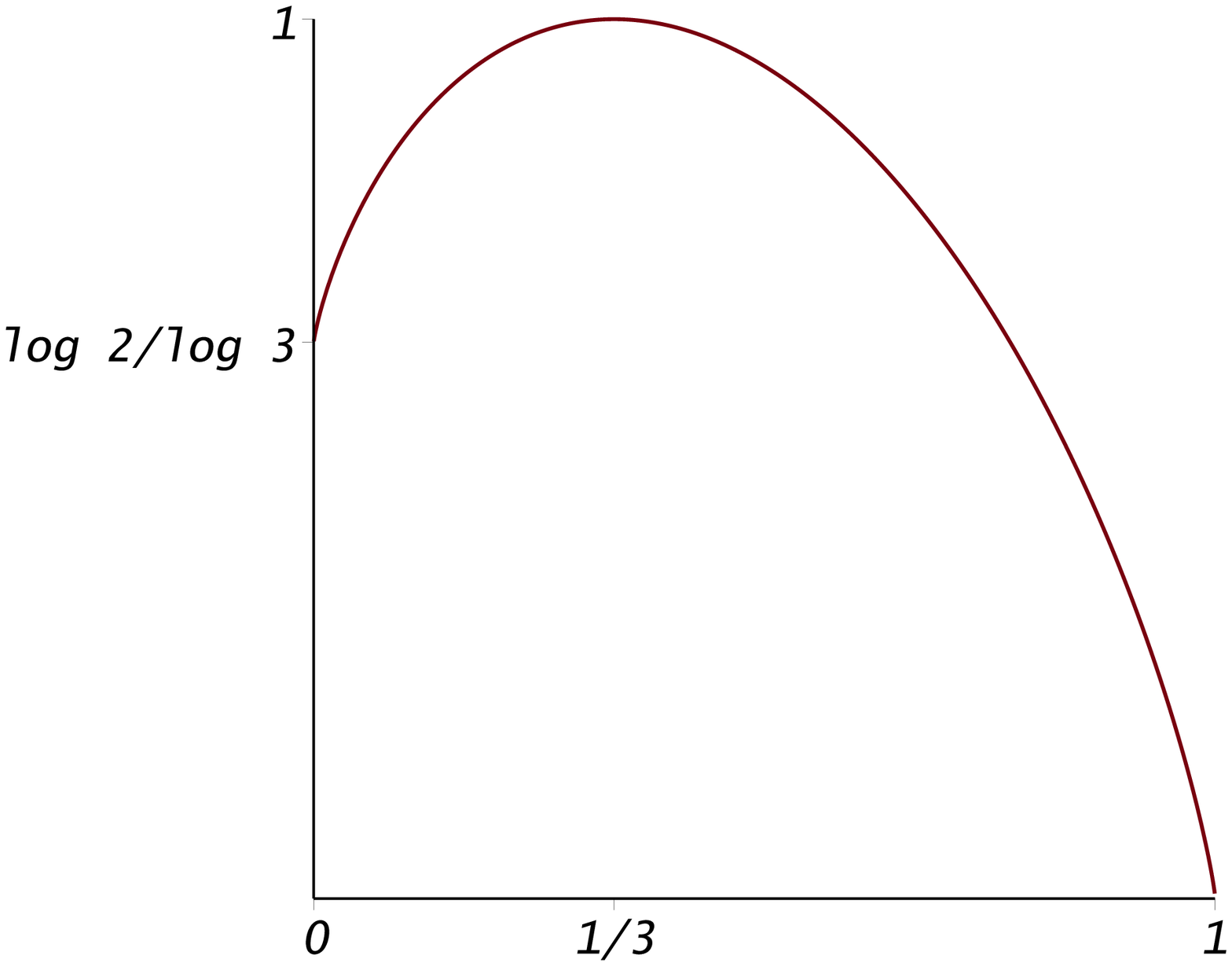, height=.2\textheight, width=.35\textwidth}
\caption{Graphs of $\phi$ (left) and $h$ (right)}
\label{fig:phi-and-h}
\end{center}
\end{figure}


\begin{theorem} \label{thm:size-and-dimension}
\begin{enumerate}[(i)]
\item The sets $\DD_0(a)$ are descending in $a$ on $(0,1/3)$, ascending on $(1/3,1/2)$, and descending on $[1/2,2/3]$. Furthermore,
\begin{equation*}
\dim_H \DD_0(a)=\begin{cases}
1 & \mbox{if $0<a\leq a_0$, $a\neq 1/3$},\\
d(a) & \mbox{if $a_0\leq a\leq 2/3$},\\
0 & \mbox{if $a\geq 2/3$}.
\end{cases}
\end{equation*}
\item The sets $\DD_\infty(a)$ are ascending in $a$ on $(0,1/3)$, descending on $(1/3,1/2]$, and descending on $(1/2,\rho]$, with a discontinuity at $1/2$ in the sense that $\DD_\infty(1/2)\not\supset\DD_\infty(a)$ for $1/2<a<\rho$. Furthermore,
\begin{equation*}
\dim_H \DD_\infty(a)=d(a), \qquad 0<a\leq 1/2, \quad a\neq 1/3.
\end{equation*}
\item The sets $\mathcal{N}(a)$ are ascending in $a$ on $(1/2,1)$, and
\begin{equation*}
\dim_H \mathcal{N}(a)=\begin{cases}
d(a) & \mbox{if $0<a\leq a_0$, $a\not\in\{1/3,1/2\}$},\\
\left(\log_3 2\right)^2 & \mbox{if $a=1/2$},\\
1 & \mbox{if $a\geq a_0$}.
\end{cases}
\end{equation*}
\end{enumerate}
\end{theorem}

\begin{figure}
\begin{center}
\epsfig{file=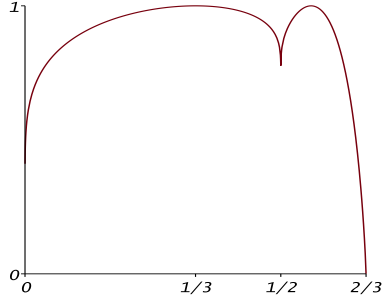, height=.2\textheight, width=.35\textwidth}
\caption{Graph of $d(a)$. Note that $d(0)=0$ and $d(1/2)=\log_3 2$.}
\label{fig:d}
\end{center}
\end{figure}

Note that $\dim_H \mathcal{N}(a)$ is discontinuous at $a=1/2$, since $d(1/2)=\log_3 2$.

It seems difficult to compute the exact Hausdorff dimension of $\DD_{\infty}(a)$ for $1/2<a<\hat{a}$. We observe here that, since $\DD_{\infty}(a)$ 
is covered by countably many affine copies of $\CC$, its dimension is at most $\log_3 2$. In the next section (see Remark \ref{eq:dimension-bounds-for-D-infty}) we will derive significantly tighter upper and lower bounds for $\dim_H\DD_{\infty}(a)$.

\bigskip
In order to prove Theorem \ref{thm:size-and-dimension}, some more notation is needed. Let
\begin{equation*}
u_1(x):=\limsup_{n\to\infty}\frac{i(n;x)}{n}, \qquad l_1(x):=\liminf_{n\to\infty}\frac{i(n;x)}{n},
\end{equation*}
for $x\in[0,1]$, where $i(n;x)$ is as defined at the beginning of Section \ref{sec:prelim}. For $p\in[0,1]$, define the sets
\begin{gather*}
R^p:=\{x\in[0,1]: u_1(x)<p\}, \qquad \bar{R}^p:=\{x\in[0,1]: u_1(x)\leq p\},\\
R_p:=\{x\in[0,1]: l_1(x)>p\}, \qquad \bar{R}_p:=\{x\in[0,1]: l_1(x)\geq p\},\\
S^p:=\{x\in[0,1]: u_1(x)>p\}, \qquad \bar{S}^p:=\{x\in[0,1]: u_1(x)\geq p\},\\
S_p:=\{x\in[0,1]: l_1(x)<p\}, \qquad \bar{S}_p:=\{x\in[0,1]: l_1(x)\leq p\}.
\end{gather*}
(Note that these sets satisfy pairwise complementary relationships, e.g. $S^p=[0,1]\backslash \bar{R}^p$, etc.)

\begin{lemma} \label{lem:upper-and-lower-density} 
We have
\begin{gather}
\dim_H R^p=\dim_H \bar{R}^p=\dim_H S_p=\dim_H \bar{S}_p=\begin{cases}
h(p) & \mbox{if $0\leq p\leq 1/3$},\\
1 & \mbox{if $1/3\leq p\leq 1$},
\end{cases}
\label{eq:small-p-dimension}
\\
\dim_H R_p=\dim_H \bar{R}_p=\dim_H S^p=\dim_H \bar{S}^p=\begin{cases}
1 & \mbox{if $0\leq p\leq 1/3$},\\
h(p) & \mbox{if $1/3\leq p\leq 1$},
\end{cases}
\label{eq:large-p-dimension}
\end{gather}
and
\begin{equation}
\dim_H (S_p\cap S^p)=\dim_H (\bar{S}_p\cap \bar{S}^p)=h(p), \qquad 0\leq p\leq 1.
\label{eq:intersection-dimension}
\end{equation}
\end{lemma}

\begin{proof} 
We first prove \eqref{eq:small-p-dimension}.
Let $N_d^{(n)}(x):=\#\{j: 1\leq j\leq n,\ \xi_j=d\}$, $d=0,1,2$. (So $N_1^{(n)}(x)=i(n;x)$.) Define the sets
\begin{equation*}
\mathcal{F}(p_0,p_1,p_2):=\left\{x\in[0,1]: \lim_{n\to\infty} n^{-1}N_d^{(n)}(x)=p_d,\ d=0,1,2\right\},
\end{equation*}
for $p_0,p_1,p_2\in[0,1]$ such that $p_0+p_1+p_2=1$. It is well known (e.g. \cite[Proposition 10.1]{Falconer}) that
\begin{equation}
\dim_H \mathcal{F}(p_0,p_1,p_2)=-\frac{1}{\log 3}\sum_{i=0}^2 p_i\log p_i.
\label{eq:entropy-formula}
\end{equation}
If $p>1/3$, then all four sets in \eqref{eq:small-p-dimension} contain the set $\mathcal{F}(1/3,1/3,1/3)$, so their Lebesgue measure is $1$ by Borel's normal number theorem. Assume now that $0<p\leq 1/3$. Since $\bar{R}^p$ contains the set
\begin{equation*}
\mathcal{F}\left(\frac{1-p}{2},p,\frac{1-p}{2}\right),
\end{equation*}
\eqref{eq:entropy-formula} gives $\dim_H \bar{R}^p\geq h(p)$, and then of course also $\dim_H \bar{S}_p\geq h(p)$. But $R^p\supset \bar{R}^{p-\eps}$ and $S_p\supset \bar{S}_{p-\eps}$ for all $\eps>0$, so by the continuity of $h$, $\dim_H R^p\geq h(p)$ and $\dim_H S_p\geq h(p)$. 

For the reverse inequality, by monotonicity of the Hausdorff dimension it is enough to show that $\dim_H \bar{S}_p\leq h(p)$. This follows from a slight modification of the proof of Proposition 10.1 in \cite{Falconer}. For a $k$-tuple $(i_1,\dots,i_k)\in\{0,1,2\}^k$, let $I_{i_1,\dots,i_k}=\{x\in[0,1]: \xi_1(x)=i_1,\dots, \xi_k(x)=i_k\}$, so $I_{i_1,\dots,i_k}$ is a triadic interval of length $3^{-k}$. For $x\in[0,1]$ and $k\in\NN$, let $I_k(x)$ be the unique interval $I_{i_1,\dots,i_k}$ which contains $x$. Define a probability measure $\mu$ on $[0,1]$ by
\begin{equation*}
\mu(I_{i_1,\dots,i_k})=p^{n_1(i_1,\dots,i_k)}\left(\frac{1-p}{2}\right)^{k-n_1(i_1,\dots,i_k)},
\end{equation*}
for each $k\in\NN$ and $(i_1,\dots,i_k)\in\{0,1,2\}^k$, where $n_1(i_1,\dots,i_k):=\#\{j: 1\leq j\leq k,\ i_j=1\}$.
Let $x\in\bar{S}_p$, and $s>h(p)$. Then
\begin{equation*}
\frac{1}{k}\log\frac{\mu(I_k(x))}{|I_k(x)|^s}=\left\{\log p-\log\left(\frac{1-p}{2}\right)\right\}\frac{i(k)}{k}+\log\left(\frac{1-p}{2}\right)+s\log 3,
\end{equation*}
where $|I_k(x)|=3^{-k}$ denotes the length of $I_k(x)$. Since $p\leq 1/3$ and $\liminf i(k)/k\leq p$, it follows that
\begin{align*}
\limsup_{k\to\infty} \frac{1}{k}\log\frac{\mu(I_k(x))}{|I_k(x)|^s}&\geq p\left\{\log p-\log\left(\frac{1-p}{2}\right)\right\}+\log\left(\frac{1-p}{2}\right)+s\log 3\\
&=\big(s-h(p)\big)\log 3>0,
\end{align*}
and hence,
\begin{equation*}
\limsup_{k\to\infty} \frac{\mu(I_k(x))}{|I_k(x)|^s}=\infty.
\end{equation*}
Thus, by the mass distribution principle (see for instance \cite[Proposition 4.9]{Falconer}, and note that balls there may be replaced by triadic intervals), $\dim_H \bar{S}_p\leq h(p)$. This concludes the proof of \eqref{eq:small-p-dimension} for $0<p\leq 1$. The case $p=0$ follows by monotonicity in $p$ of the sets involved and the continuity of $h$. The proof of \eqref{eq:large-p-dimension} is analogous.

As for \eqref{eq:intersection-dimension}, note first that \eqref{eq:small-p-dimension} and \eqref{eq:large-p-dimension} immediately give the upper bound
\begin{equation*}
\dim_H (\bar{S}_p\cap \bar{S}^p)\leq\min\{\dim_H \bar{S}_p,\dim_H \bar{S}^p\}=h(p).
\end{equation*}
To establish the lower bound, define the sets
\begin{equation*}
\mathcal{E}_p^q:=\{x\in[0,1]: l_1(x)=p,\ u_1(x)=q\}, \qquad 0<p\leq q<1.
\end{equation*}
A modification of the proof of Theorem 6 of Carbone et al. \cite{Carbone} yields
\begin{equation}
\dim_H \mathcal{E}_p^q=\min\{h(p),h(q)\}.
\label{eq:Carbone-formula}
\end{equation}
Since $S_p\cap S^p\supset \mathcal{E}_{p-\eps}^{p+\eps}$ for each $\eps>0$, this implies, by the continuity of $h$, that
\begin{equation*}
\dim_H (S_p\cap S^p)\geq h(p).
\end{equation*}
This completes the proof, because $S_p\cap S^p\subset \bar{S}_p\cap \bar{S}^p$.
\end{proof}

\begin{proof}[Proof of Theorem \ref{thm:size-and-dimension}]
(i) Proposition \ref{prop:zero-derivative}(i) implies that
\begin{gather}
R^{\phi(a)}\subset \DD_0(a)\subset \bar{R}^{\phi(a)}, \qquad 0<a<1/3, \label{eq:D0-lower}\\
R_{\phi(a)}\subset \DD_0(a)\subset \bar{R}_{\phi(a)}, \qquad 1/3<a<2/3, \quad a\neq 1/2. \label{eq:D0-upper}
\end{gather}
(Note that Okamoto \cite[Remark 1]{Okamoto} incorrectly states (in our notation) that $S_{\phi(a)}\subset \DD_0(a)$ for $0<a<1/3$.) Since the sets $R^p$ and $\bar{R}^p$ are ascending in $p$ and $\phi$ is strictly decreasing on $[0,1/2]$, it follows from \eqref{eq:D0-lower} that $\DD_0(a)$ is descending in $a$ on $(0,1/3)$. Likewise, since $R_p$ and $\bar{R}_p$ are descending in $p$, \eqref{eq:D0-upper} yields that $\DD_0(a)$ is ascending on $(1/3,1/2)$. But $\phi$ is strictly increasing on $[1/2,2/3]$, so $\DD_0(a)$ is descending on $(1/2,2/3]$. Finally, we may include the left endpoint $1/2$ in this last interval since
\begin{equation}
\DD_0(a)\subset [0,1]\backslash\CC=\DD_0(1/2), \qquad a>1/2,
\end{equation}
a consequence of the fact that $f_n^+(x)=(3a)^n\to\infty$ for $x\in\CC$ and $a>1/2$. The only dimension statement in (i) that requires an argument is that $\dim_H \DD_0(a)=d(a)$ for $a_0\leq a\leq 2/3$; this follows from \eqref{eq:D0-upper} and Lemma \ref{lem:upper-and-lower-density} in view of the monotonicity of Hausdorff dimension.

(ii) From Proposition \ref{prop:zero-derivative}(ii) we obtain that
\begin{gather}
R_{\phi(a)}\subset \DD_\infty(a)\subset \bar{R}_{\phi(a)}, \qquad 0<a<1/3, \label{eq:Dinf-lower}\\
R^{\phi(a)}\subset \DD_\infty(a)\subset \bar{R}^{\phi(a)}, \qquad 1/3<a<1/2. \label{eq:Dinf-upper}
\end{gather}
These inclusions imply, via an argument similar to the one in the proof of part (i) above, that $\DD_\infty(a)$ is ascending in $a$ on $(0,1/3)$ and descending on $(1/3,1/2)$. Since by Proposition \ref{thm:triadic}(ii),
\begin{equation}
\DD_\infty(1/2)\subset \CC\backslash\mathcal{T} \subset \DD_\infty(a), \qquad 1/3<a<1/2,
\end{equation}
it follows that $\DD_\infty(a)$ is in fact descending on $(1/3,1/2]$. That $\DD_\infty(a)$ is descending on $(1/2,\rho]$ follows from Remark \ref{rem:simpler-condition}. But $\DD_\infty(a)$ is not descending on the entire interval $(1/3,\rho]$, since $\DD_\infty(1/2)$ does not contain any points with at least one `$1$' in their ternary expansion, whereas $\DD_\infty(a)$ contains infinitely many such points when $1/2<a<\rho$.

That $\dim_H D_\infty(a)=d(a)$ for $a\in(0,1/2)\backslash\{1/3\}$ follows from \eqref{eq:Dinf-lower}, \eqref{eq:Dinf-upper} and Lemma \ref{lem:upper-and-lower-density}. Finally, that $\dim_H \DD_\infty(1/2)=d(1/2)=\log_3 2$ follows since Theorem \ref{thm:main}(ii) and the Borel-Cantelli lemma imply that $\mu(\DD_\infty(1/2))=1$, where $\mu$ is the Cantor measure, determined by $\mu([0,x])=F_{1/2}(x)$ for $x\in[0,1]$.

(iii) Taking complements in \eqref{eq:D0-upper} and using Theorem \ref{thm:main}(i), we have
\begin{equation}
S_{\phi(a)}\cap S^0\subset \mathcal{N}(a)\subset \bar{S}_{\phi(a)}, \qquad 1/2<a<2/3,
\label{eq:N-upper}
\end{equation}
which shows that $\mathcal{N}(a)$ is ascending in $a$ on $(1/2,2/3)$. For $a\geq 2/3$, $\mathcal{N}(a)=(0,1)$, so $\mathcal{N}(a)$ is actually ascending on $[1/2,1)$. The dimension of $\mathcal{N}(1/2)$ was computed by Darst \cite{Darst}. That $\dim_H \mathcal{N}(a)=d(a)$ for $a\in(1/2,a_0)$ follows from \eqref{eq:N-upper} and Lemma \ref{lem:upper-and-lower-density}, noting that $\phi(a_0)=1/3$. (For the lower estimate, observe that $S_p\cap S^0\supset\{x\in[0,1]: l_1(x)=u_1(x)=p-\eps\}$ for $0<\eps<p<1$, and use \eqref{eq:Carbone-formula} and the continuity of $h$.) For $a\in(0,1/2)\backslash\{1/3\}$, the same expression follows from \eqref{eq:intersection-dimension} and the inclusions
\begin{equation}
S_{\phi(a)}\cap S^{\phi(a)}\subset \mathcal{N}(a)\subset \bar{S}_{\phi(a)}\cap\bar{S}^{\phi(a)}, \qquad 0<a<1/2,\quad a\neq 1/3,
\end{equation}
obtained by taking complements in \eqref{eq:D0-lower}, \eqref{eq:D0-upper}, \eqref{eq:Dinf-lower} and \eqref{eq:Dinf-upper}.
\end{proof}

\begin{remark}
{\rm
For $a<1/2$, $F_a$ belongs to the class of functions considered by Jordan et al. \cite{JKPS}. Their Theorem 1.1 implies immediately that $\dim_H \DD_\infty(a)=\dim_H \mathcal{N}(a)$, and gives an implicit formula for the value of this dimension in terms of pressure functions. However, it seems difficult to obtain the dimension explicitly from their formula as this involves solving a transcendental equation. For this specific case, considering the simple self-affine structure of $F_a$, our approach above is easier and quite natural.
}
\end{remark}

\section{Beta-expansions and the size of $\DD_\infty(a)$} \label{sec:beta-expansions}

The purpose of this section is to prove Theorem \ref{cor:golden-ratio}, and to examine the set $\DD_\infty(a)$ in more detail when $1/2<a<\rho$. 
We will mostly work on the symbol space $\Omega:=\{0,1\}^\NN$. Denote a generic element of $\Omega$ by $\omega=(\omega_1,\omega_2,\dots)$. We equip $\Omega$ with the family of metrics $\{\varrho_\lambda\}_{0<\lambda<1}$ defined by $\varrho_\lambda(\omega,\eta)=\lambda^{\inf\{n:\omega_n\neq\eta_n\}}$. Since the Hausdorff dimension of a subset of $\Omega$ depends on the metric used, we will let $\dim_H^{(\lambda)}E$ denote the Hausdorff dimension of $E\subset\Omega$ induced by the metric $\varrho_\lambda$. It is straightforward to verify that, for $0<\lambda_1,\lambda_2<1$, 
\begin{equation}
\dim_H^{(\lambda_1)}E=\frac{\log{\lambda_2}}{\log{\lambda_1}}\dim_H^{(\lambda_2)}E, \qquad E\subset\Omega.
\label{eq:Hausdorff-dimension-conversion}
\end{equation}

Let $\sigma$ denote the (left) shift map on $\Omega$: $\sigma(\omega)=(\omega_2,\omega_3,\dots)$. For $0<\lambda<1$ and $\omega\in\Omega$, let
\begin{equation*}
\Pi_\lambda(\omega):=\sum_{n=1}^\infty \omega_n \lambda^n.
\end{equation*}
Let a bar denote reflection: $\bar{0}=1$, $\bar{1}=0$, and for $\omega=(\omega_1,\omega_2,\dots)\in\Omega$, $\bar{\omega}=(\bar{\omega}_1,\bar{\omega}_2,\dots)$. Define the sets
\begin{equation*}
\UU_\lambda:=\{\omega\in\Omega: \Pi_\lambda(\sigma^k(\omega))<1\ \mbox{and}\ \Pi_\lambda(\sigma^k(\bar{\omega}))<1\ \mbox{for all $k\in\ZZ_+$}\},
\end{equation*}
and
\begin{equation*}
\widetilde{\UU}_\lambda:=\bigcup_{\delta>0}\widetilde{\UU}_{\lambda,\delta},
\end{equation*}
where
\begin{equation*}
\widetilde{\UU}_{\lambda,\delta}:=\{\omega\in\Omega: \Pi_\lambda(\sigma^k(\omega))<1-\delta\ \mbox{and}\ \Pi_\lambda(\sigma^k(\bar{\omega}))<1-\delta\ \mbox{for all $k\in\ZZ_+$}\}.
\end{equation*}
Let $\Phi:\Omega\to\CC$ be given by
\begin{equation*}
\Phi(\omega):=2\Pi_{1/3}(\omega), \qquad\omega\in\Omega.
\end{equation*}
Finally, introduce the family of affine maps
\begin{equation*}
\psi_{n,k}(x):=3^{-n}(x+k), \qquad n\in\NN, \quad k=0,1,\dots,3^n-1.
\end{equation*}
It follows from Theorem \ref{thm:main}(i) that
\begin{equation}
\bigcup_{n,k}\psi_{n,k}\big(\Phi\big(\widetilde{\UU}_a\big)\big)\subset \DD_\infty(a)\subset \bigcup_{n,k}\psi_{n,k}(\Phi(\UU_a)),
\label{eq:key-sandwich}
\end{equation}
where the union is over $n\in\NN$ and $k=0,1,\dots,3^n-1$. Since Hausdorff dimension is countably stable and unaffected by affine transformations, it is therefore enough to investigate the cardinality and Hausdorff dimension of the sets $\UU_a$ and $\widetilde{\UU}_a$. For this we can use the theory of $\beta$-expansions (e.g. \cite{GlenSid,JSS,Parry}). For $1<\beta<2$ and a real number $0<x<1$, a {\em $\beta$-expansion} of $x$ is a representation of the form 
\begin{equation}
x=\sum_{n=1}^\infty \omega_n\beta^{-n}=\Pi_{1/\beta}(\omega), 
\label{eq:beta-expansion}
\end{equation}
where $\omega=(\omega_1,\omega_2,\dots)\in\Omega$. In general, $\beta$-expansions are not unique. The {\em greedy} $\beta$-expansion of $x$ is the lexicographically largest $\omega$ satisfying \eqref{eq:beta-expansion}, which chooses a $1$ whenever possible; and the {\em lazy} expansion is the lexicographically smallest such $\omega$, which chooses a $0$ whenever possible. A number $x$ has a unique $\beta$-expansion if its greedy and lazy $\beta$-expansions are the same.

Let $1/2<\lambda<1$ and $\beta=1/\lambda$. Let $\VV_\lambda$ be the set of $\omega\in\Omega$ such that 
\begin{equation*}
\frac{2\lambda-1}{1-\lambda}<\Pi_\lambda(\omega)<1
\end{equation*}
and $\Pi_\lambda(\omega)$ has a unique $\beta$-expansion. Note that for such $\omega$, $\Pi_\lambda(\bar{\omega})$ also lies in $((2\lambda-1)/(1-\lambda),1)$, since $\Pi_\lambda(\omega)+\Pi_\lambda(\bar{\omega})=\lambda/(1-\lambda)$. Let $1=\sum_{n=1}^\infty d_n\beta^{-n}$ be the greedy $\beta$-expansion of $1$; but if there is an $n$ such that $d_n=1$ and $d_j=0$ for all $j>n$, we replace $(d_j)$ by the sequence $(d_j'):=(d_1\dots d_{n-1}0)^\infty$ and rename this new sequence again as $(d_j)$. Put $d=(d_1,d_2,\dots)$. It is well known (e.g. \cite[Lemma 4]{GlenSid}) that
\begin{equation*}
\VV_\lambda=\{\omega\in\Omega: \sigma^k(\omega)\prec d\ \mbox{and}\ \sigma^k(\bar{\omega})\prec d\ \mbox{for all}\ k\in\ZZ_+\},
\end{equation*}
where $\prec$ denotes the (strict) lexicographic order on $\Omega$.

\begin{lemma} \label{lem:equivalence}
Let $1/2<\lambda<1$. Then $\UU_\lambda=\VV_\lambda$.
\end{lemma}

\begin{proof}
Let $\lambda$, $\beta$ and $d$ have the relationships outlined above.
The lemma will follow once we establish the equivalence
\begin{equation}
\Pi_\lambda(\sigma^k(\omega))<1\quad\forall\, k\in\ZZ_+ \qquad \Longleftrightarrow \qquad \sigma^k(\omega)\prec d \quad\forall\, k\in\ZZ_+.
\label{eq:shift-equivalenvce}
\end{equation}
Assume that $\Pi_\lambda(\omega)<1$, and suppose that $\omega\succeq d$. Since $\Pi_\lambda(d)=1$ by definition, $\omega\neq d$ and hence there is $n\in\NN$ such that $\omega_1\dots\omega_{n-1}=d_1\dots d_{n-1}$ and $\omega_{n}=1$, $d_n=0$. Define now the finite sequence $(\tilde{d}_j)_{j=1}^n$ by $\tilde{d}_j=d_j$ for $j=1,\dots,n-1$, and $\tilde{d}_n=1$. Then $(\tilde{d}_j)$ can be extended to a (nonterminating) $\beta$-expansion of $1$ which is greater than $d$ in the lexicographic order. This contradicts $d$ being the greedy expansion of $1$. Thus, $\omega\prec d$. Since this argument holds for arbitrary $\omega\in\Omega$, the forward direction of \eqref{eq:shift-equivalenvce} follows. The converse is proved in \cite[Lemma 1]{Parry}.
\end{proof}

The next lemma is the key to the proof of Theorem \ref{cor:golden-ratio}.

\begin{lemma}[Glendinning and Sidorov \cite{GlenSid}] \label{lem:Glenn-Sid}
The set $\VV_\lambda$ is countable for $\lambda>\hat{a}$, but has positive Hausdorff dimension for $1/2<\lambda<\hat{a}$.
\end{lemma}

The next two lemmas collect some more useful results from the literature. They are due to Jordan et al. \cite{JSS}, whose primary objective was to analyze the multifractal spectrum of Bernoulli convolutions.

\begin{lemma}[\cite{JSS}, Lemma 2.3] \label{lem:JSS-sandwich}
If $\lambda_1<\lambda_2$, then $\UU_{\lambda_1}\supset \widetilde{\UU}_{\lambda_1}\supset \UU_{\lambda_2}$.
\end{lemma}

\begin{lemma}[\cite{JSS}, Lemma 2.7]  \label{lem:bi-Lipschitz}
For $\lambda>1/2$, the restriction of $\Pi_\lambda$ to $\UU_\lambda$ is bi-Lipschitz with respect to the metric $\varrho_\lambda$.
\end{lemma}

Let $\mathcal{A}_\lambda:=\Pi_\lambda(\UU_\lambda)$. The following crucial result, which was already stated in \cite{GlenSid} without proof, was proved only very recently by Komornik et al. \cite{KKL}.

\begin{lemma}[\cite{KKL}, Theorems 1.3, 1.4] \label{lem:dimension-continuity}
\begin{enumerate}[(i)]
\item Let $N_n(\lambda)$ denote the number of words in $\{0,1\}^n$ that can be extended to a sequence in $\UU_\lambda$. The limit
\begin{equation}
h(\UU_\lambda):=\lim_{n\to\infty}\frac{\log N_n(\lambda)}{n}
\label{eq:entropy-definition}
\end{equation}
exists, and
\begin{equation*}
\dim_H \mathcal{A}_\lambda=-\frac{h(\UU_\lambda)}{\log\lambda}.
\end{equation*}
\item The function $\lambda \mapsto \dim_H \mathcal{A}_\lambda$ is continuous in $\lambda$ on $1/2<\lambda<\hat{a}$.
\end{enumerate}
\end{lemma} 

\begin{proof}[Proof of Theorem \ref{cor:golden-ratio}]
Let $\hat{a}<a<\rho$. Then by Lemma \ref{lem:equivalence}, Lemma \ref{lem:Glenn-Sid} and \eqref{eq:key-sandwich}, $\DD_\infty(a)$ is countable. In fact, we can give a very explicit description of $\DD_\infty(a)$ in this case. For $n\in\NN$, let $\hat{a}_n$ be the root in $(1/2,1)$ of $\sum_{j=1}^{2^n}t_j a^j=1$, where $(t_j)$ is the Thue-Morse sequence; see \eqref{eq:Thue-Morse}. Then $\hat{a}_1=\rho$ and $\hat{a}_n\searrow \hat{a}$ as $n\to\infty$, so for given $a\in(\hat{a},\rho)$, there is $n\in\NN$ such that $a\in[\hat{a}_{n+1},\hat{a}_n)$. As shown in \cite[Proposition 13]{GlenSid}, $\UU_a$ then contains only sequences ending in $(v_m\bar{v}_m)^\infty$ for some $m<n$, where $v_m=t_1\dots t_{2^m}$. Since such sequences lie in $\widetilde{\UU}_a$ if they lie in $\UU_a$, it follows that in fact $\widetilde{\UU}_a=\UU_a$. We now see from \eqref{eq:key-sandwich} that $\DD_\infty(a)$ consists exactly of those points whose ternary expansions are obtained by taking an arbitrary sequence from $\Omega$ ending in $(v_m\bar{v}_m)^\infty$ for some $m<n$, replacing all $1$'s by $2$'s, and appending the resulting sequence to an arbitrary finite prefix of digits in $\{0,1,2\}$. In particular, $\DD_\infty(a)$ is countably infinite and contains only rational points.


Next, let $1/2<a<\hat{a}$. Combining Lemmas \ref{lem:JSS-sandwich}, \ref{lem:bi-Lipschitz} and \ref{lem:dimension-continuity} it follows that
\begin{equation*}
\dim_H^{(\lambda)}\widetilde{\UU}_\lambda=\dim_H^{(\lambda)}\UU_\lambda=-\frac{h(\UU_\lambda)}{\log\lambda},
\end{equation*}
and this dimension is continuous in $\lambda$. Moreover, by Lemmas \ref{lem:equivalence} and \ref{lem:Glenn-Sid}, it is strictly positive for $1/2<\lambda<\hat{a}$. By \eqref{eq:Hausdorff-dimension-conversion} we then have 
\begin{equation}
\dim_H^{(1/3)}\widetilde{\UU}_\lambda=\dim_H^{(1/3)}\UU_\lambda=\frac{h(\UU_\lambda)}{\log 3}.
\label{eq:H-dim-formula-one-third}
\end{equation}
Now $\Pi_{1/3}$ is bi-Lipschitz even on the full domain $\Omega$ with respect to the metric $\varrho_{1/3}$ (the proof of this is essentially the same as that of \cite[Lemma 2.7]{JSS}), so \eqref{eq:H-dim-formula-one-third} and \eqref{eq:key-sandwich} give
\begin{equation}
\dim_H \DD_\infty(a)=\frac{h(\UU_a)}{\log 3}.
\label{eq:final-dimension-formula}
\end{equation}
This shows that $\dim_H \DD_\infty(a)$ is strictly positive and continuous in $a$ on $1/2<a<\hat{a}$. That it is nonincreasing in $a$ is immediate from Theorem \ref{thm:size-and-dimension}(ii). The final statement of Theorem \ref{cor:golden-ratio}, that $\dim_H \DD_\infty(a)$ decreases in the manner of a devil's staircase, follows from \eqref{eq:final-dimension-formula} in conjunction with Theorems 2.5 and 2.6 of Kong and Li \cite{KongLi}, which imply the existence of a countable collection $\{I_j\}_{j\in\NN}$ of disjoint subintervals of $(1/2,\hat{a})$ such that $\bigcup_j I_j$ has full Lebesgue measure in $(1/2,\hat{a})$ and $h(\UU_\lambda)$ is constant on $I_j$ for each $j$. (It is the topological entropy of a certain subshift of finite type.) 
\end{proof}

\begin{remark} \label{rem:critical-value}
{\rm
It is shown in \cite{GlenSid} that $\UU_{\hat{a}}$ is uncountable with zero Hausdorff dimension. This implies that $\dim_H\DD_\infty(\hat{a})=0$, but it remains unclear whether $\DD_\infty(\hat{a})$ is countable or uncountable.
}
\end{remark}

\begin{remark} \label{eq:dimension-bounds-for-D-infty}
{\rm
In principle, using \eqref{eq:final-dimension-formula} and \eqref{eq:entropy-definition} the Hausdorff dimension of $\DD_\infty(a)$ can be estimated to any desired accuracy for any given $a\in(1/2,\hat{a})$. But a closed-form expression in terms of $a$ appears to be out of reach. However, we can obtain fairly tight and simple bounds for $\dim_H \DD_\infty(a)$ as follows. For $k\in\NN$, let $a_k$ be the root in $(1/2,1]$ of $\sum_{j=1}^k a^j=1$ (so $a_1=1$, $a_2=\rho$, and more generally, $a_k$ is the $k$th multinacci number). Note that $a_k\searrow 1/2$, so for $a\in(1/2,\hat{a})$ there is $k$ such that $a\in[a_{k+1},a_k)$. Let $\mathcal{Q}_k$ be the set of sequences in $\Omega$ that do not contain $1^k$ or $0^k$ as a sub-word. It is not difficult to see that
\begin{equation}
a\in[a_{k+1},a_k) \qquad\Longrightarrow \qquad \mathcal{Q}_k\subset \widetilde{\UU}_a\subset\UU_a\subset \mathcal{Q}_{k+1}.
\label{eq:double-sandwich}
\end{equation}
(To see the first inclusion, note that the sequence in $\mathcal{Q}_k$ with the largest value under $\Pi_a$ is $\omega:=(1^{k-1}0)^\infty$, and $\Pi_a(\omega)=(a+a^2+\dots+a^{k-1})/(1-a^k)<1$.) The Hausdorff dimension of $\mathcal{Q}_k$ can be calculated exactly: it is
\begin{equation*}
\dim_H^{(1/3)} \mathcal{Q}_k=\frac{-\log(a_{k-1})}{\log 3}, \qquad k\geq 2.
\end{equation*}
(This can be seen, for instance, by using the graph directed construction of Mauldin and Williams \cite{MW}; alternatively, see \cite[Example 17]{GlenSid} for a sketch of a proof.) It therefore follows from \eqref{eq:key-sandwich}, \eqref{eq:double-sandwich} and the bi-Lipschitz property of $\Phi|_{\UU_a}$ under the metric $\varrho_{1/3}$, that
\begin{equation*}
a\in[a_{k+1},a_k) \qquad\Longrightarrow \qquad \frac{-\log(a_{k-1})}{\log 3}\leq \dim_H \DD_\infty(a)\leq \frac{-\log(a_{k})}{\log 3}.
\end{equation*}
Since $a_k$ converges to $1/2$ very rapidly, these bounds are quite tight even for moderate values of $k$. Moreover, they show that $\dim_H \DD_\infty(a)$ is continuous at $a=1/2$ (see Theorem \ref{thm:size-and-dimension}(ii)), and also that $\dim_H \DD_\infty(a)<\log_3 2=\dim_H\{x:f_n'(x)\to\pm\infty\}$ when $a>1/2$.

}
\end{remark}

\section{The case of rational $x$} \label{sec:rational}

In this final section we examine what the condition in Theorem \ref{thm:main}(i) means for (nontriadic) rational $x$. To keep the presentation simple, we consider only points in $\CC$, which have a ternary expansion with $\xi_n\in\{0,2\}$ for all $n$. The straightforward generalization to arbitrary rational points is left to the reader. For $x\in\QQ\cap(0,1)$, there exists $m\in\NN$ such that the ternary expansion $\{\xi_n\}$ of $x$ satisfies $\xi_{k+m}=\xi_k$ for all sufficiently large $k$; call the smallest such $m$ the {\em period} of $\{\xi_n\}$. 

\begin{theorem} \label{thm:rational-case}
Let $x\in \QQ\cap \CC$ have ternary expansion $\{\xi_n\}$ with period $m\geq 2$. Write $x$ as $x=0.\xi_1\dots\xi_{k_0}(\zeta_1\dots\zeta_m)^\infty$, where $k_0$ is chosen so that $\zeta:=\zeta_1\dots\zeta_m$ is lexicographically largest among all its cyclical permutations. For $j=1,\dots,m$, set $\eta_j:=\zeta_j/2$. Then $\eta_m=0$, and $F_a^+(x)=\infty$ if and only if 
\begin{equation}
\sum_{j=1}^{m-1} \eta_j a^j+a^m<1.
\label{eq:polynomial-inequality}
\end{equation}
\end{theorem}

\begin{proof}
That $\eta_m=0$ is an immediate consequence of $\zeta_1\dots\zeta_m$ being the lexicographically largest cyclical permutation of the period of $\{\xi_n\}$. Condition \eqref{eq:polynomial-inequality} is necessary because there exist infinitely many $n\in\NN$ such that
\begin{equation*}
\sum_{k=1}^\infty \delta_2(\xi_{n+k})a^k=\sum_{j=1}^m \eta_j a^j(1+a^m+a^{2m}+\dots)=\frac{1}{1-a^m}\sum_{j=1}^{m-1} \eta_j a^j.
\end{equation*}
Sufficiency follows from the ideas of the previous section, in particular the equivalence \eqref{eq:shift-equivalenvce}. If we have \eqref{eq:polynomial-inequality}, then $\eta_1\dots\eta_{m-1}1\prec d_1\dots d_m$, where $1=\sum_{n=1}^\infty d_n a^n$ is the greedy expansion of $1$ in base $\beta:=1/a$. But then $p:=(\eta_1\dots\eta_m)^\infty\prec d$, and since $\eta$ is lexicographically largest among its cyclical shifts, it follows that $\sigma^k(p)\prec d$ for all $k\in\ZZ_+$. Thus, $\Pi_a(\sigma^k(p))<1$ for all $k\in\ZZ_+$. This implies that 
\begin{equation*}
\limsup_{n\to\infty} \sum_{k=1}^\infty a^{k}\delta_2(\xi_{n+k})<1,
\end{equation*}
and hence (see Remark \ref{rem:simpler-condition}), that $F_a^+(x)=\infty$.
\end{proof}

Recall that $F_a^-(x)=\infty$ if and only if $F_a^+(1-x)=\infty$, so whether $F_a'(x)=\infty$ can be determined by applying Theorem \ref{thm:rational-case} first to $x$ and then to $1-x$.

\begin{example} \label{ex:rational}
{\rm
Let $x=0.0220(2000202)^\infty$. Then $m=7$, and the lexicographically largest cyclical permutation of the repeating part is $\zeta=2200020$, so $\eta=1100010$. Thus, $F_a^+(x)=\infty$ if and only if $a+a^2+a^6+a^7<1$. On the other hand, the $m$-tuple $\eta$ corresponding to $1-x=0.2002(0222020)^\infty$ is $\eta=1110100$, so $F_a^-(x)=\infty$ if and only if $a+a^2+a^3+a^5+a^7<1$. The latter condition is more stringent, so $F_a'(x)=\infty$ if and only if $1/3<a<a^*$, where $a^*\approx .5261$ is the unique positive root of $a+a^2+a^3+a^5+a^7=1$. 
}
\end{example}


\section*{Acknowledgment}
I am grateful to the anonymous referee for many valuable suggestions and for pointing out the interesting works mentioned in the next-to-last paragraph of the Introduction. I thank Dr. Lior Fishman for raising the question about the Hausdorff dimension of the exceptional sets in Okamoto's theorem, which eventually led to Theorem \ref{thm:size-and-dimension}. Finally, I am indebted to Dr. Derong Kong for sending me the papers \cite{KKL} and \cite{KongLi}.

\end{document}